\documentclass[12pt,reqno,a4paper]{amsart}
\usepackage{amsfonts}
\usepackage{epsfig}
\usepackage{graphicx}
\usepackage{amsmath}
\usepackage{amssymb}

\newcounter{main}

\numberwithin{equation}{section}

\newtheorem{theorem}{Theorem}[section]
\newtheorem{proposition}[theorem]{Proposition}
\newtheorem{lemma}[theorem]{Lemma}
\newtheorem{corollary}[theorem]{Corollary}

\newtheorem{remark}{Remark}[section]
\newtheorem{definition}{Definition}[section]
\newtheorem{maintheorem}{Theorem}

\newcommand{\blanksquare}{\,\,\,$\sqcup\!\!\!\!\sqcap$}

\newcounter{example}
{{\stepcounter{example}}{\flushleft {\bf Example \arabic{example}:}}}%
{\par}

\title[Geometric and ergodic theory of conservative flows]
{Contributions to the geometric and ergodic theory of conservative flows}

\author[M. Bessa]{M\'{a}rio Bessa}
\author[J. Rocha]{Jorge Rocha}
\address{Departamento de Matem\'atica Pura da Faculdade de Ci\^encias da Universidade do Porto, 
Rua do Campo Alegre, 687, 
4169-007 Porto, Portugal}
\email{bessa@fc.up.pt}
\email{jrocha@fc.up.pt}

\date{21/10/2008}

\begin{document}

\begin{abstract}
We  prove the following dichotomy for vector fields in a $C^1$-residual subset of volume-preserving flows: 
for Lebesgue almost every point all Lyapunov exponents equal to zero  or   its orbit has a dominated splitting. As a consequence if we have a vector field in this residual that cannot be $C^1$-approximated by a vector field having elliptic periodic orbits, then, there exists a full measure set such that every orbit of this set admits a dominated splitting for the linear Poincar\'e flow. Moreover, we prove that a volume-preserving and $C^1$-stably ergodic flow can be $C^1$-approximated by another volume-preserving flow which is non-uniformly hyperbolic. 
\end{abstract}

\maketitle

\emph{keywords:} Volume-preserving flows; Lyapunov exponents; Dominated splitting; Stable ergodicity.\\

\begin{section}{Introduction and statement of the results}
Let $M$ be a $d$-dimensional, $d\geq 3$, compact, connected and boundaryless Riemaniann manifold endowed with a volume-form $\omega$ and let $\mu$ denote the Lebesgue measure associated to it. We denote by $\mathfrak{X}^{1}_\mu(M)$ the space of $C^1$ vector fields $X$ over $M$ such that $X$ is \emph{divergence-free}, that is its associated flow  $X^t$ preserves the measure $\mu$. We consider $\mathfrak{X}^{1}_\mu(M)$ endowed with the usual Whitney $C^1$-topology. 

Given a flow $X^t$ one usually deduces properties of it by studying its linear approximation. One way to do that is by considering the Lyapunov exponents which, in broad terms, detect if there are any exponential behavior of the linear tangent map along orbits. Given $X\in\mathfrak{X}^{1}_\mu(M)$ the existence of Lyapunov exponents for almost every point is guaranteed by  Oseledets' theorem (\cite{O}). Positive (or negative) exponents assure, in average,  exponential rate of divergence (or convergence) of two neighboring trajectories, whereas zero exponents give us the lack of any kind of average exponential behavior. A flow is said to be \emph{nonuniformly hyperbolic} if its Lyapunov exponents are all different from zero. In ~\cite{HPT} Hu, Pesin and Talitskaya gave examples of nonuniformly hyperbolic flows in any manifold. Non-zero exponents plus some smoothness assumptions on the flow allows us to obtain invariant manifolds dynamically defined (see~\cite{P}). Since this stable/unstable manifold theory is the base of capital results on dynamical systems nowadays it is of extreme importance to detect when we have nonzero Lyapunov exponents. 

In the beginning of the 1980' Ricardo Ma\~n\'e, in~\cite{M1}, announced a dichotomy for $C^1$-generic discrete-time conservative systems which in broad terms says that for Lebesgue almost every point its Lyapunov exponents are all equal to zero or else there exists a weak form of uniform hyperbolicity along its orbit.

It is well-known that hyperbolicity plays a crucial role if one wants to obtain stability. Briefly speaking, hyperbolicity means uniform expansion (or contraction) by the tangent map along the orbits and when restricted to particular invariant subbundles. A quintessential example is an Anosov flow~\cite{An}.

By a weak form of hyperbolicity we mean uniformly contraction of the ratio between the dynamical behavior of the tangent map when computed in an invariant subbundle and the dynamical behavior of the tangent map restricted to another invariant subbundle which is most contracting (or less expanding) than the first mentioned.

Later, in~\cite{M2}, Ma\~n\'e presented the guidelines for the proof of the aforementioned dichotomy in the surfaces case. However, it was ne\-cessary more ingredients and new tools to obtain a complete proof (see the work of Bochi~\cite{Bo}). Then, in a remarkable paper~\cite{BV}, Bochi and Viana extended the Bochi-Ma\~n\'e theorem to any dimensional manifolds and recently Bochi (see~\cite{Bo2}) was able to obtain the full statement announced in~\cite{M1} for the symplectomorphisms setting.

For the flow setting the first author proved in~\cite{B0} the three-dimensional version for vectors fields without equilibrium points and also a weak version for general divergence-free vector fields. Later, in ~\cite[Theorem A]{AB}, a global version for vector fields with equilibrium points was obtained. In ~\cite{B1} was proved a version for linear differential systems with conservativeness properties and in~\cite{BD} was obtained a similar result in the Hamiltonians setting. 

After the perturbation techniques developed in~\cite{B0} and in~\cite{B1}  we expected to obtain Bochi-Viana's theorem for a $C^1$-dense subset of $\mathfrak{X}^{1}_\mu(M)$, however by an upgrade refinement on the perturbation framework we were able to obtain this result for a $C^1$-residual subset of $\mathfrak{X}^{1}_\mu(M)$, thus achieve the full counterpart of~\cite[Theorem 2]{BV}.

More precisely we prove the following result.

\begin{maintheorem}\label{T1}
There exists $C^1$-residual set $\mathcal{E} \subset  \mathfrak{X}^{1}_\mu(M)$ such that if $X \in \mathcal{E}$ then there exist two $X^t$-invariant subsets of $M$, $\mathcal{Z}$ and  $\mathcal{D}$, whose union has full measure and such that:
\begin{itemize}
\item if $p \in \mathcal{Z}$ then all the Lyapunov exponents associated to $p$ are zero;
\item if $p \in \mathcal{D}$ then its orbit admits a dominated splitting for the linear Poincar\'e flow.
\end{itemize}
\end{maintheorem}

We point out that the  abundance of zero exponents from the generic point of view seems to be strongly related to the topology used, namely to the $C^1$-topology. On the other hand recent results obtained by Viana (\cite{V}) show that, in a prevalent way, H\"older continuous linear cocycles based on a uniformly hyperbolic system with local product structure have nonzero Lyapunov exponents.

Recall that $X\in\mathfrak{X}^{1}_\mu(M)$ is \emph{ergodic} if any measurable $X^t$-invariant set is a  zero measure set or is a full measure set. We say that $X\in\mathfrak{X}^{1}_\mu(M)$ is a $C^1$-\emph{stably ergodic} flow if there exists a $C^1$-neighborhood of $X$ such that any $Y\in\mathfrak{X}^{1}_\mu(M)$ is ergodic. 

Let us denote by $\mathcal{SE}^1$ the space of the $C^1$-stably ergodic flows in $\mathfrak{X}^{1}_\mu(M)$. We refer the reader to the survey of Pugh and Shub (\cite{PS}) on properties of these systems. From Theorem~\ref{T1} it follows that if $X \in \mathcal{E }\cap \mathcal{SE}^1$ then either the set $\mathcal{Z}$ has full measure or else $\mathcal{D}$ is a full measure set. In the next result we prove that, for an open and dense subset of 
$\mathcal{SE}^1$, actually $\mathcal{Z}$ has zero measure and the pointwise domination given by Theorem~\ref{T1} for points in the full measure set $\mathcal{D}$ is in fact uniform. This result is a (strong) continuous-time version  of Bochi-Fayad-Pujals Theorem for conservative diffeomorphism (\cite{BFP}).

\begin{maintheorem}\label{T2}
There exists $C^1$-open and dense set $\mathcal{U} \subset \mathcal{SE}^1$ such that if $X \in \mathcal{U}$ then $X^t$ is a non-uniformly hyperbolic flow and $X$ admits a dominated splitting.
\end{maintheorem}

Next theorem shows that if a $X \in \mathcal{E}$ does not belong to the closure of those having elliptic closed orbits then it exhibits some kind of weak hyperbolicity. More precisely we obtain the following result. 

\begin{maintheorem}\label{T3} Let $X \in \mathcal{E}$ such that $X$ cannot be $C^1$-approximated by a vector field $Y$ having elliptic periodic orbits. Then, there exists a full measure set $N\subset M$ such that  every orbit of $N$ admits a dominated splitting for the linear Poincar\'e flow.
\end{maintheorem}

As a consequence of the previous result we obtain the following corollary.

\begin{corollary}\label{C1}  Let $X \in \mathcal{E}$ such that $Sing(X)=\emptyset$ and $X$ cannot be $C^1$-approximated by a vector field $Y$ having elliptic periodic orbits. Then for every $\epsilon>0$ there exist $m \in \mathbb{N}$ and a closed and invariant set $N_\epsilon \subset M$ such that $\mu(N_\epsilon)> 1-\epsilon$ and $N_\epsilon$ admits a $m$-dominated splitting for the linear Poincar\'e flow of $X$.
\end{corollary}

\end{section}

We observe that if $M$ is a three-dimensional manifold then $N_\epsilon$ is a hyperbolic set.
We also notice that,  from the proof of this corollary, if we remove the hypothesis on the nonexistence of singularities then we obtain an invariant set $N_\epsilon$ with the same properties except that we can not assure its compacity. 

\begin{section}{Notation, definitions and basic results}

In this section we introduce some notation, fundamental definitions and basic results needed to prove our theorems.
 
\begin{subsection}{The setting}
Let $M$ be a $d$-dimensional compact, connected and boundaryless Riemaniann manifold endowed with a volume-form $\omega $. We call Lebesgue measure to the measure $\mu$ associated to $\omega$. As mentioned before we denote by $\mathfrak{X}^{1}_{\mu}(M)$ the set of all \emph{divergence-free} vector fields $X\colon M\rightarrow TM$ of class $C^1$, endowed  with the usual Whitney $C^{1}$-topology.. Given $X\in\mathfrak{X}^{1}_{\mu}(M)$ let $X^{t}$ be its infinitesimal generator, that is, $\frac{dX^{t}}{dt}|_{t=s}(p)=X(X^{s}(p))$. We are interested in the study of the tangent map $DX^{t}_{p}\colon T_{p}M\rightarrow T_{X^{t}}M$. Notice that $DX^{t}_{p}$ is a solution of the linear
variational equation $\dot{u}(t)=DX_{X^{t}(p)}\cdot u(t)$. Is is easy to see that $|\det(DX^{t})|=1$ for any $t \in \mathbb{R}$, that is the flow $X^t$ is volume-preserving.

Let $Sing(X):=\{x\in M\colon X(x)=\vec{0}\}$ denote the set of \emph{singularities} of $X$ and let $\mathcal{R}(X):=M\setminus Sing(X)$ denote the set of \emph{regular} points.

\end{subsection}

\begin{subsection}{Linear Poincar\'e flow}

Fix $X\in\mathfrak{X}^{1}_{\mu}(M)$, $p\in \mathcal{R}(X)$ and let $N_{p}\subset T_{p}M$ denote the \emph{normal fiber} at $X(p)$, that is, the subfiber spanned by the orthogonal complement of $X(p)$. We denote by $N\subset TM$ the \emph{normal bundle} which, of course, is only defined on $\mathcal{R}(X)$. Now, let $\mathcal{N}_{p}$ and $\mathcal{N}_{X^{t}(p)}$ be two $(d-1)$-dimensional manifolds contained in $M$ whose tangent spaces at $p$ and $X^{t}(p)$, respectively, are $N_{p}$ and $N_{X^{t}(p)}$. Let also $\mathcal{V}_{p}$ be a small neighborhood of $p$ in $\mathcal{N}_{p}$. If $\mathcal{V}_{p}$ can be taken small enough, then the usual \emph{Poincar\'{e} map} $\mathcal{P}_{X}^{t}(p):\mathcal{V}_{p}\subset\mathcal{N}_{p}\rightarrow\mathcal{N}_{X^{t}(p)}$ is well defined.

The \emph{linear Poincar\'{e} flow}  was formally introduced in~\cite{D} and it is the differential of the Poincar\'{e} map. To define it properly for each $t \in \mathbb{R}$ we consider the tangent map $DX^{t}:T_{R}M\longrightarrow T_{R}M$ which is defined by $DX^{t}(p,v)=(X^{t}(p),DX^{t}(p)\cdot v)$ and let $\Pi_{X^{t}(p)}$ be the canonical projection on
$N_{X^{t}(p)}$. The linear map $P^{t}_{X}(p):N_{p}\longrightarrow N_{X^{t}(p)}$ defined by $P^{t}_{X}(p)=\Pi_{X^{t}(p)}DX^{t}(p)$ is called the linear
Poincar\'{e} flow at $p$ associated to the vector field $X$. 

\end{subsection}

\begin{subsection}{Local coordinates}\label{lc}
Given a linear map $A$ we define its norm in the usual way, i.e., $$\sup_{v\not=\vec{0}}\frac{\|A\cdot v\|}{\|v\|}.$$ 

By Lemma 2 of~\cite{Mo}, given a volume form $\omega$ in $M$ there exists an atlas $\mathcal{A}=\{(\alpha_{i}, U_{i})_i\}$ of $M$, such that $(\alpha_{i})_{*}\omega=dx_{1}\wedge{dx_{2}}\wedge...\wedge{dx_{d}}$. The fact that $M$ is compact guarantees that $\mathcal{A}$ can be taken finite. 

The Riemannian norm \emph{a priori} fixed at $TM$ will not be used, instead we use the equivalent norm $\|v\|_{x}:=\|(D\alpha_{i})_{x}\cdot v\|$.

Let $p$ and $q$ be points in the same domain $U_i$ and $t$ be such that $X^{t}(p)$ and $X^{t}(q)$ are in the same domain $U_j$. Given linear maps $A^t(p) \colon N_{p}\rightarrow{N_{X^{t}(p)}}$ and $B^t(q)\colon N_{q}\rightarrow{N_{X^{t}(q)}}$ we define the distance $\|A^t(p)-B^t(q)\|$ in the following way. Let
\begin{itemize}
\item $a^t_{i,j}=(D\alpha_{j})_{X^{t}(p)}|_{N_ {X^t(p)}} \circ A^t(p) \circ (D\alpha_{i})_{p}^{-1}|_{D\alpha_i(N_p)}$,   and
\item $b^t_{i,j}=(D\alpha_{j})_{X^{t}(q)}|_{N_ {X^t(q)}} \circ B^t(q) \circ (D\alpha_{i})_{q}^{-1}|_{D\alpha_i(N_q)}$.
\end{itemize}
Now we define
\begin{equation}\label{metrics} 
\|A^t(p)-B^t(q)\|=\|a^t_{i,j}-b^t_{i,j}\|.
\end{equation}

\end{subsection}

\begin{subsection}{Flowboxes and the modified volume-preserving property}

Given the Poincar\'{e} map of a non-periodic point, $\mathcal{P}_{X}^{t}(p):\mathcal{V}_{p}\subseteq N_{p}\rightarrow{N_{X^{t}(p)}}$, where $\mathcal{V}_{p}$ is chosen sufficiently small and given $B\subseteq\mathcal{V}_{p}$ the self-disjoint set $$\mathcal{F}_{X}^{n}(p)(B):=\{\mathcal{P}_{X}^{t}(p)(q):q\in{B},t\in[0,n]\},$$ is called the time-$n$ length  \emph{flowbox} at $p$ associated to the vector field $X$.

Given
$v_{1},v_{2},...,v_{d-1}\in{N_{p}}$  we can define a pair of $(d-1)$-forms by $$\hat{\omega}_{p}(v_{1},v_{2},...,v_{d-1}):=\omega_{p}(X(p),v_{1},v_{2},...,v_{d-1}),$$ and $$\overline{\omega}_{p}(v_{1},v_{2},...,v_{d-1})=\omega_{p}(\|X(p)\|^{-1}X(p),v_{1},v_{2},...,v_{d-1}),$$ both induced by the volume form
$\omega$. It turns out that
 $(\mathcal{P}_{X}^{t}(p))^{*}\hat{\omega}_{p}=\hat{\omega}_{X^t(p)}$. The measure $\overline{\mu}$ 
induced by the $(d-1)$-form $\overline{\omega}$ is not necessarily
$\mathcal{P}_{X}^{t}$-invariant, however both the associated measures $\hat{\mu}$ and $\overline{\mu}$ are equivalent. We call $\overline{\mu}$ the Lebesgue measure at normal
sections or \emph{modified section volume}. In fact, given $v_{1},v_{2},...,v_{d-1}\in{N_{p}}$ we have that 
$$(\mathcal{P}_{X}^{t}(p))^{*}\overline{\omega}_{p}(v_{1},...,v_{d-1})=x(t)^{-1}\overline{\omega}_{X^t(p)}(P_X^t(p)\cdot v_{1},...,P_X^t(p)\cdot v_{d-1}),$$ where $x(t)=\|X(X^{t}(p))\|\|X(p)\|^{-1}$. Since the flow is volume-preserving we have $|\det P_{X}^{t}(p)|=x(t)^{-1}$.
Therefore it follows that we can give an explicit expression for the infinitesimal
distortion volume factor of the linear Poincar\'{e} flow, which is expressed by the following simple lemma (\cite{B0}).

\begin{lemma}\label{map}
Given $\nu>0$ and $T>0$, there exists $r>0$ such that for any
measurable set $K\subseteq{B(p,r)}\subseteq{\mathcal{N}_{p}}$ we have
$$\|X(p)\| |\overline{\mu}(K)-x(t).\overline{\mu}(\mathcal{P}_{X}^{t}(p)(K))|<\nu,\,\text{for all }\, t\in[0,T].$$

\end{lemma}

\end{subsection}

\begin{subsection}{Multiplicative ergodic theorem for the linear Poincar\'{e} flow}\label{MET}

Let $$\mathbb{R}X(p):=\{v\in T_{p}M:v=\eta X(p), \eta\in\mathbb{R}\},$$ be the vector field direction at $p$. Recalling  that $DX^{t}_{p}(X(p))=X(X^{t}(p))$ we conclude that the vector field direction is $DX^{t}$-invariant. The existence of other $DX^{t}$-invariant fibers is guaranteed, at least for Lebesgue almost every point, by  a theorem due to  Oseledets (see~\cite{O}) that we re-write for the linear Poincar\'e flow. 

\begin{theorem}(Oseledets' Theorem for the linear Poincar\'{e} flow)\label{Oseledec}
Given $X\in{\mathfrak{X}^{1}_{\mu}(M)}$, then for $\mu$-a.e. $p\in{M}$
there exist
\begin{itemize}
 \item a $P_{X}^{t}$-invariant splitting of the fiber $N_{p}=N_{p}^{1}\oplus ... N_{p}^{k(p)}$ along the orbit of $p$ (\emph{Oseledets' splitting}) and
 \item real numbers $\hat{\lambda}_{1}(p)>...>\hat{\lambda}_{k(p)}(p)$ (\emph{Lyapunov exponents}), 
\end{itemize}
with $1\leq k(p)\leq d-1$, such that:

\begin{equation}\label{limit}
\underset{t\rightarrow{\pm{\infty}}}{\lim}\frac{1}{t}\log{\|P_{X}^{t}(p)\cdot n^{i}\|=\hat{\lambda}_{i}(p)}, 
\end{equation}
for any $n^{i}\in{N^{i}_{p}\setminus\{\vec{0}\}}$ and $i=1,...,k(p)$.  
\end{theorem}

Let $\mathcal{O}(X)$ denote the set of $\mu$-generic points given by this theorem. We note that if we do not take into account the multiplicities of the $\hat{\lambda}_{i}(p)$, then we have $d-1$ Lyapunov exponents: $\lambda_{1}(p)\geq \lambda_{2}(p)\geq...\geq\lambda_{d-1}(p)$.

In ~\cite{JPS} is presented a  proof of Theorem~\ref{Oseledec} in the context of linear differential systems.

\begin{remark}
Actually, the Oseledets Theorem gives us a splitting of $T_{p}M=E_{p}^{1}\oplus ... \oplus E_{p}^{k(p)}\oplus \mathbb{R}X(p)$ and Lyapunov exponents associated to these $DX^t$-invariant directions for $\mu$-a.e. point $p$.  Due to the fact that for any of these subspaces $E^{i}_{p}\subset T_{p}M$, the angle between this space and  $\mathbb{R}X(p)$ along the orbit has sub-exponential growth, that is
\begin{equation}\label{angle}
\lim_{t\rightarrow{\pm{\infty}}}\frac{1}{t}\log\sin(\measuredangle(E^{i}_{X^{t}(p)},\mathbb{R}X(X^t(p))))=0.
\end{equation}
we conclude that the Lyapunov exponent $\hat{\lambda}_{i}(p)$ for $DX^t$ with associated subspace $E^i_p$ is also a Lyapunov exponent  for $P_{X}^t$ associated to subspace $N^{i}_{p}=\Pi_p(E^i_p)$, $i \in\{1,...,k(p)\}$, where  $\Pi_p$ is the projection into $N_{p}$.

\end{remark}

\end{subsection}

\begin{subsection}{Multilinear algebra for the linear Poincar\'{e} flow} The $k^{th}$ exterior product of $N$, denoted by $\wedge^{k}(N)$, is a $\binom{d-1}{k}$-dimensional vector space. Let $\{e_{j}\}_{j\in J}$ be an orthonormal basis of $N$, then the family of exterior products $e_{j_{1}}\wedge e_{j_{2}}\wedge...\wedge e_{j_{k}}$ for $j_{1}<...<j_{k}$ with $j_{\alpha}\in J$ forms an orthonormal basis of $\wedge^{k}(N)$. Given $P_{X}^{t}(p)\colon N_{p}\rightarrow N_{X^{t}(p)}$ we define 
$$\begin{array}{cccc}
\wedge^{k}(P_{X}^t(p))\colon & \wedge^{k}(N_{p}) & \longrightarrow  & \wedge^{k}(N_{X^{t}(p)}) \\
& \psi_{1}\wedge...\wedge\psi_{k} & \longrightarrow  & P_{X}^{t}(p)\cdot\psi_{1}\wedge...\wedge P_{X}^{t}(p)\cdot\psi_{k}.
\end{array}$$

This formalism of multilinear algebra reveals to be the adequate to prove our results. This is because we can recover the spectrum and the splitting information of the dynamics of $\wedge^{k}(P_{X}^t(p))$ from the one obtained by applying Oseledets' Theorem to $P_{X}^t(p)$. This is precisely the meaning of the next theorem (\cite[Theorem 5.3.1]{A}).


\begin{theorem}(Oseledets' Theorem for the exterior power of the linear Poincar\'{e} flow)\label{arnauld}
The Lyapunov exponents $\lambda_{i}^{\wedge k}(p)$ for $i\in\{1,...,(_{k}^{d-1})\}$ (repeated 
 with multiplicity) of the $k^{th}$ exterior product operator $\wedge^{k}(P_{X}^t(p))$ are the numbers of the form:
$$\sum_{j=1}^{k}\lambda_{i_{j}}(p), \text { where }1\leq i_{1}< ...<i_{k}\leq d-1.$$ 
This nondecreasing sequence starts with 
\begin{itemize}
 \item $\lambda_{1}^{\wedge k}(p)=\lambda_{1}(p)+\lambda_{2}(p)+...+\lambda_{k}(p)$ and ends with \item $\lambda_{q(k)}^{\wedge k}(p)=\lambda_{d-k}(p)+\lambda_{d+1-k}(p)+...+\lambda_{d-1}(p)$.
\end{itemize}
Moreover, the splitting of $\wedge ^{k}(N_{p}(i))$ for $0\leq i\leq q(k)$, associated to  $\wedge^{k}(P_{X}^{t}(p))$ and to $\lambda_{i}^{\wedge k}(p)$,  can be obtained from the splitting $N_{p}(i)$ of $P_{X}^{t}(p)$ as follows. Take an Oseledets basis $\{e_{1}(p),...,e_{d-1}(p)\}$ of  $N_{p}$ such that $e_{i}(p)\in E_{p}^{\ell}$ for $$\dim(E_{p}^{1})+...+\dim(E_{p}^{\ell-1})<i\leq \dim(E_{p}^{1})+...+\dim(E_{p}^{\ell}).$$ Then the Oseledets space is generated by the $k$-vectors:
$$e_{i_{1}}\wedge ...\wedge e_{i_{k}}\text { such that } 1\leq i_{1}<...<i_{k}\leq d-1 \text { and } \sum_{j=1}^{k}\lambda_{i_{j}}(p)=\lambda_{i}^{\wedge k}(p).$$
\end{theorem}

\end{subsection}

\begin{subsection}{Dominated splitting (or projective hyperbolicity) for the linear Poincar\'{e} flow}

Let $\mathfrak{m}(A)=\|A^{-1}\|^{-1}=\inf_{v\not=\vec{0}}\|A\cdot v\|$ denotes the co-norm of a linear map $A$. 

Take a $X^t$-invariant set $\Lambda$ and fix $m\in\mathbb{N}$. A nontrivial $P_{X}^{t}$-invariant and continuous splitting $N_{\Lambda}=U_{\Lambda}\oplus S_{\Lambda}$ is said to have an $m$\emph{-dominated splitting} for the linear Poincar\'{e} flow of $X$ over $\Lambda$ if the following inequality holds for every $p\in\Lambda$:
\begin{equation}\label{ds}
 \frac{\|P_{X}^{m}(p)|_{S_{p}}\|}{\mathfrak{m}(P_{X}^{m}(p)|_{U_{p}})}\leq \frac{1}{2}.
\end{equation}
The \emph{index} of the splitting is the dimension of the bundle $U_\Lambda$. The dominated splitting structure is a ``weak'' form of uniform hyperbolicity, in fact behaves like an uniform hyperbolic structure in the projective space $RP^{d-2}$. We enumerate some basic properties of an $m$-dominated splitting on a set $\Lambda$, for the detailed proofs of these properties see~\cite{BDV} Section B.1.
\begin{enumerate}
 \item [(H)] (Higher Domination) There exists $m_0>m$ such that, for all $\ell \geq m_0$, $U_{\Lambda}\oplus S_{\Lambda}$ is an $\ell$-dominated splitting.
 \item [(E)] (Extension) It can always be extended to an $m$-dominated splitting over  $\overline{\Lambda} \setminus Sing(X)$.
 \item [(T)] (Transversality) The angles between the $U_p$ and $S_p$ are uniformly bounded away from zero, for $p \in \overline{\Lambda}$.
 \item [(U)] (Uniqueness) For a fixed index the dominated splitting is unique.
\end{enumerate}

Fix $X\in\mathfrak{X}^{1}_{\mu}(M)$, $k \in\{1,...,d-2\}$ and $m\in\mathbb{N}$. The subset of $M$ formed by the points $p\in M$ such that there exists an $m$-dominated splitting of index $k$ along the orbit of $p$ is denoted by $\Lambda_{k}(X,m)$. The set $\Gamma_{k}(X,m)= M\setminus \overline{\Lambda_{k}(X,m)}$ is open and each  element of it has an iterate without $m$-dominated splitting of index $k$ or else it is a singularity of $X$. From the generic point of view the set of (hyperbolic) singularities is a measure zero set (\cite{Rb}).

Along this paper we will be mainly interested in dominated splitting related to the natural $P_{X}^t$-invariant splitting given by Oseledets' Theorem (obtained in Subsection~\ref{MET}) and over the orbit of some $p\in\mathcal{O}(X)$, namely,
$$
 U_{p}^j=N^{1}_{p}\oplus ... \oplus N_{p}^{j} \text{ and } S_{p}^j=N^{j+1}_{p}\oplus ... \oplus N_{p}^{\ell},
$$
where $\ell \leq d-1$, $N^{i}_{p}\subset N_{p}$ for $i=\{1,...,\ell\}$  and $j$  is some fixed index of the splitting, $j\in\{1,...,d-2\}$. 

Now, we define some sets which will be used in the sequel:
\begin{itemize}
 \item $\Gamma_{k}^{\sharp}(X,m):=\{p\in \Gamma_{k}(X,m)\cap \mathcal{O}(X)\colon \lambda_{k}(X,p)>\lambda_{k+1}(X,p)\}$;
 \item $\Gamma_{k}^{*}(X,m):=\Gamma_{k}^{\sharp}(X,m)\setminus Per(X)$, where $Per(X)$ denotes the periodic points of $X^t$, for all $t$;
 \item $\Gamma_{k}(X,\infty):=\underset{m\in\mathbb{N}}{\bigcap}{\Gamma_{k}(X,m)}$ and
 \item $\Gamma_{k}^{\sharp}(X,\infty):=\underset{m\in\mathbb{N}}{\bigcap}{\Gamma_{k}^{\sharp}(X,m)}$.
\end{itemize}

The next lemma (Lemma 4.1 of~\cite{BV}) allows us to focus our attention only on the non-periodic points.

\begin{lemma}
Using the same notation as above, for every $\delta>0$, there exists $m_{0}\in\mathbb{N}$ such that for all $m\geq m_{0}$ we have that $$\mu(\Gamma_{k}^{\sharp}(X,m)\setminus \Gamma_{k}^{*}(X,m))<\delta.$$ 
\end{lemma}

We consider a measurable function $\rho_{X,m}\colon \Gamma_{k}^{*}(X,m) \longrightarrow  \mathbb{R}$ defined by $\rho_{X,m}(p)=\frac{\|P_{X}^{m}(p)|_{S_{p}}\|}{\mathfrak{m}(P_{X}^{m}(p)|_{U_{p}})}$. It is clear that if $p\in\Gamma_{k}^{*}(X,m)$, then there exists $t \in \mathbb{R}$ satisfying the inequality $\rho_{X,m}(X^t(p))>\frac{1}{2}$. We define the set $\Delta_{k}^{*}(X,m)$ by those points in $\Gamma_{k}^{*}(X,m)$ such that $\rho_{X,m}(p)>\frac{1}{2}$. Of course that $\Gamma_{k}^{*}(X,m)$ is the superset of $\Delta_{k}^{*}(X,m)$ saturated by the flow, i.e., $\Gamma_{k}^{*}(X,m)=\bigcup_{t\in\mathbb{R}}X^{t}(\Delta_{k}^{*}(X,m))$. 

In ~\cite{B1}  (Lemma 2.2) is proved the following result relating the measures of these two sets.
\begin{lemma}\label{invariance}
Given $\Delta_{k}^{*}(X,m)$ and $\Gamma_{k}^{*}(X,m)$ as above, if $\mu(\Gamma_{k}^{*}(X,m))>0$, then $\mu(\Delta_{k}^{*}(X,m))>0$.
\end{lemma}

\end{subsection}

\end{section}

\begin{subsection}{The integrated upper Lyapunov exponent of exterior power of the linear Poincar\'{e} flow}\label{EF}
We consider the following function:
\begin{equation}\label{entropy}
\begin{array}{cccc} 
LE_{k}\colon & \mathfrak{X}^{1}_{\mu} & \longrightarrow & [0,+\infty) \\
& X & \longmapsto & \int_{M}\lambda_{1}(\wedge^{k}(X),p) d\mu(p).
\end{array} 
\end{equation}
In the same way we define the function $LE_{k}(X,\Gamma)$, where $\Gamma\subseteq M$ is a $X^{t}$-invariant set, defined by:
$$LE_{k}(X,\Gamma)=\int_{\Gamma}\lambda_{1}(\wedge^{k}(X),p) d\mu(p).$$
Let $\Sigma_{k}(X,p)$ denotes the sum of the $k$ first Lyapunov exponents of $X$, that is  $\Sigma_{k}(X,p)=\lambda_{1}(X,p)+...+\lambda_{k}(X,p)$. Is is an easy consequence of Theorem~\ref{arnauld} that for $k=1,...,d-2$ we have $\Sigma_{k}(X,p)=\lambda_{1}(\wedge^{k}(X),p)$ and so $LE_{k}(X,\Gamma)=LE_{1}(\wedge^{k}(X),\Gamma)$, for any $X^t$-invariant set $\Gamma$. By using Proposition 2.2 of~\cite{BV} we get immediately that:
\begin{equation}
LE_{k}(X,\Gamma)=\underset{j\in\mathbb{N}}{\inf}\frac{1}{j}\int_{\Gamma}\log\|\wedge^{k}(P_{X}^{j}(p))\|d\mu(p),
\end{equation}
concluding that, for all $k\in\{1,...,d-2\}$, the function~(\ref{entropy}) is an upper semicontinuous function.

\end{subsection}

\begin{section}{Proof of Theorem ~\ref{T1}} \label{PT1}

Next we consider an abstract object called \emph{realizable linear flow} which will play a central role in the proofs of the main theorems. Briefly, it consists in the following: we want to change the action of the linear Poincar\'e flow along the orbit of a given Oseledets point with lack of hyperbolic behavior, in order to decay its exponential asymptotic behavior. However, one single point is meaningless since we consider the Lebesgue measure and, moreover, the chosen  point may not be an Oseledets point for the perturbed vector field. So, in broad terms, we consider time-$t$ modified volume-preserving linear maps acting in the normal fiber at $p$,
$L_{t}(p)\colon N_{p}\rightarrow{N_{X^{t}(p)}}$ which perform exactly the action that we
want. Then, we build a divergence-free vector field, $C^1$-close to the original one such that the time-$t$ linear Poincar\'e map at $q\in K$ of this new vector field has \emph{almost} the same behavior as the map $L_{t}(p)$, where $K$ is a measurable set contained in a pre-assigned open set inside a small transversal section of $p$ such that both these sets have almost the same measure.

With this definition in mind we are able to ``realize dynamically'' the perturbations did in ~\cite{B1} concerning the skew-product flows version of Theorem~\ref{T1}.

This may be seen as a Franks' Lemma (see \cite{F} or \cite{AP,BGV,BR2} for the flows version) of a measure theoretical flavor type.

\begin{definition}\label{rlf}
Given $X\in{\mathfrak{X}^{1}_{\mu}(M)}$, $\epsilon>0$, 
$\kappa \in (0,1)$, $\ell  \in \mathbb{N}$, and a non-periodic point $p$, we say that the
modified volume-preserving sequence of linear maps $L_{j}:N_{X^{j}(p)}\rightarrow{N_{X^{j+1}(p)}}$ for $j=0,...,\ell-1$ is an $(\epsilon,\kappa)$-\emph{realizable linear flow of length $\ell$ at $p$} if the following occurs.

For all $\gamma>0$, there is $r>0$ such that for any non-empty open set
$U\subseteq{B(p,r)}\subseteq{\mathcal{N}_{p}}$, there exist a measurable set $K\subseteq{U}$ and a divergence-free vector field $Y$ satisfying:
\begin{enumerate}
\item [(a)]  $\overline{\mu}(K)>(1-\kappa)\overline{\mu}(U)$; 
\item [(b)] $\|Y-X\|_{C^1}<\epsilon$;
\item [(c)] $Y^{t}=X^{t}$ outside $\mathcal{F}^{\ell}_{X}(p)(U)$ and $DX_{q}=DY_{q}$ for every $q\in U \cup \mathcal{P}_{X}^{\ell}(p)(U)$; and 
\item [(d)] If $q\in{K}$, then $\|P^{1}_{Y}(Y^{j}(q))-L_{j}\|<\gamma$ for
$j=0,1,...,\ell-1.$
\end{enumerate}

\end{definition}

\bigskip

Let us consider some easy observations about this definition.

\begin{remark}
We note that this definition  only requires  $C^{1}$-closeness of vector fields. Also observe that realizable linear flows of length $\ell \in \mathbb{R}$ are also allowed and are defined in the obvious way. 
\end{remark}

\begin{remark}\label{vitali}
By basic Vitali's covering arguments we only have to prove the realizability of the linear maps for open sets $U=B(p',r')$ where $U\subseteq{B(p,r)}$.
\end{remark}

\begin{remark}\label{concat}
It is obvious that the time-$t$ linear Poincar\'e flow is itself $(\epsilon,\kappa)$-realizable of length $t$ for
every $\epsilon$ and $\kappa$. 
\end{remark}

\begin{remark} \label{conc} Condition (c) in the Definition~\ref{rlf} enables to concatenate an $(\epsilon,\kappa_1)$-realizable linear flow of length $\ell_1$ at $p$ with an $(\epsilon,\kappa_2)$-realizable linear flow of length $\ell_2$ at $X^{\ell_1}(p)$, obtaining an $(\epsilon,\kappa_{1}+\kappa_{2})$-realizable linear flow of length $\ell_1+ \ell_2$ at $p$. Notice that if $\kappa_1+\kappa_2 \geq 1$ then we do not have any useful estimate for the measure of the set $K=K_1 \cap X^{-\ell_1}(K_2)$. Actually we are just interested in the concatenation of $b$ realizable linear flows in the cases where $\sum_{j=0}^b \kappa_j <1$ (in fact close to zero) obtaining a measurable set $K$ such that the measure of $U \setminus K $ is less than  $\sum_{j=0}^b \kappa_j$.
\end{remark}

Next proposition is a key result that allows us to mix Oseledets' directions in the absence of domination.  Once we get this result the two lemmas of this section and the proof of Theorem~\ref{T1} are obtained borrowing ~\cite{BV}.

\begin{proposition}\label{exchange}
Given $X\in{\mathfrak{X}^{1}_{\mu}(M)}$, $\epsilon>0$ and $\kappa \in (0,1)$ there exists $m_{0}\in\mathbb{N}$ such that for every $m\geq m_{0}$  the following property holds. 

For all non-periodic point $p$ with a splitting $S_{X^t(p)}^j\oplus U_{X^t(p)}^j$, $t \in \mathbb{R}$ and $j \in \{1,...,d-2\}$ fixed, satisfying
\begin{equation}\label{nodom}
\frac{\|P_{X}^{m}(p)|_{S_{p}^j}\|}{\mathfrak{m}(P_{X}^{m}(p)|_{U_{p}^j})}\geq \frac{1}{2}, 
\end{equation}
there exist $(\epsilon_i,\kappa_i)$-realizable linear flows of length $\ell_i \geq  1$ at $X^{\tau_i}(p)$, with $0 \leq i \leq  b \leq m$, $\sum_{i=0}^{b}\ell_i=m$ and $\tau_i=\sum_{j=1}^{i-1}\ell_j$, denoted by $\{L_{i}\}_{i=1}^{b}$ and vectors $\mathfrak{u}\in U_{p}^j\setminus \{\vec{0}\}$ and $\mathfrak{s}\in S_{X^m(p)}^j\setminus \{\vec{0}\}$ such that:
\begin{enumerate}
 \item [(a)] The concatenation of the $b$ realizable linear flows is an $(\epsilon,\kappa)$-realizable linear flow of length $m$ at $p$ and 
 \item [(b)] $L_{b}\circ ... \circ L_{1}(\mathfrak{u})=\mathfrak{s}$
\end{enumerate}

\end{proposition}

This proposition will be proved in sections~\ref{few} and ~\ref{lots}. In Section~\ref{few} we consider the easiest case, that is when we only need to do, at most, two perturbations to achieve our goal. Section~\ref{lots} is technically harder, because we need to do many perturbations along the orbit and each time we concatenate two realizable linear flows the relative measure in $U$ of the associated set $K$ decreases.

Once we are able to realize dynamically the action which mix the Oseledets directions,  to prove Theorem~\ref{T1} we need to use the following two results.

\begin{lemma}\label{local}(Local)
Let $X\in{\mathfrak{X}^{1}_{\mu}(M)}$. Then, given any $\epsilon,\delta>0$, a small $\kappa>0$ and $k \in \{1,\dots,d-2\}$, there exist $m_0 \in \mathbb{N}$ and, for each $m \geq m_0$, a measurable function $\tilde{T} \colon \Gamma_k^\ast(X,m) \rightarrow \mathbb{R}$  satisfying the following properties: for $\mu$-almost every point $q \in \Gamma_k^\ast(X,m)$ and every $t>\tilde{T}(q)$ there exists a modified volume-preserving sequence of linear maps $L_{j}:N_{X^{j}(q)}\rightarrow{N_{X^{j+1}(q)}}$ for $j=0,...,t-1$ which is an $(\epsilon,\kappa)$-realizable linear flow of length $t$ at $q$ satisfying  
$$\frac{1}{t}\log\|\wedge^k( L_{t-1}\circ \cdots L_1 \circ L_0)\|<\delta+\frac{1}{2}(\Sigma_{k-1}(X,q)+\Sigma_{k+1}(X,q)).$$
\end{lemma}

This lemma corresponds to Proposition 4.2 of~\cite{BV} adapted to the flow setting and, in view of Proposition~\ref{exchange}, its proof follows exactly as the proof of Lemma 4.2 of~\cite{B1}. This local procedure uses the lack of domination and also the different Lyapunov exponents to cause a decay of the largest Lyapunov exponent of the $k^{th}$ exterior power of the linear Poincar\'{e} flow by a just small perturbation.

The next lemma is a global version of the previous one. Once we have the local version (Lemma~\ref{local}) its proof follows directly the proof of Proposition 4.17 of~\cite{BV} and which uses a Kakutani's tower argument. We also refer ~\cite{B0} for the ingredients used in the flow framework. 

\begin{lemma}\label{continuous}(Global)
Let $X\in{\mathfrak{X}^{1}_{\mu}(M)}$. Then, given any $\epsilon,\delta>0$, and $k \in \{1,\dots,d-2\}$, there exists $Y\in{\mathfrak{X}^{1}_{\mu}(M)}$, $\epsilon$ $C^1$-close, such that
$$ \int_M\Sigma_k(Y,p)d\mu(p)<\int_M\Sigma_k(X,p)d\mu(p)- 2J_k(X)+\delta, $$
where $J_k(X)=\int_{\Gamma_k(X,\infty)}\lambda_k(X,p)-\lambda_{k+1}(X,p)d\mu(p)$. 
\end{lemma}

Now to prove Theorem~\ref{T1} we argue exactly as in ~\cite[pp 1467]{BV}.

For $k \in \{1,...,d-2\}$ let $\mathcal{E}_k$ be the subset of $\mathfrak{X}^{1}_{\mu}(M)$ corresponding to the points of continuity of the map $LE_{k}$, see (\ref{entropy}), and define $\mathcal{E}=\cap_1^{d-2}\mathcal{E}_k$. It is well known that the sets $\mathcal{E}_k$ are residual and so is $\mathcal{E}$. If $X \in \mathcal{E}_k$ then, by the definition of this set and by Lemma ~\ref{continuous}, $J_k(X)=0$. Therefore $\lambda_k(X,p)=\lambda_{k+1}(X,p)$ for a.e $p \in \Gamma_k(X, \infty)$. For  $X \in \mathcal{E}$ let: 
\begin{itemize}
 \item $\mathcal{Z}= \mathcal{O}(X) \cap (\cap_{k=1}^{d-2} \Gamma_k(X,\infty))$ and 
 \item $\mathcal{D}= \mathcal{O}(X)  \setminus (\cap_{k=1}^{d-2} \Gamma_k(X,\infty))$
\end{itemize}

If $p \in \mathcal{Z}$ then all the Lyapunov exponents of $p$ are equal to zero. On the other hand if $ p \in \mathcal{D}$ then $p \notin \Gamma_k(X,\infty)$ for some $1 \leq k \leq d-2$, therefore, by the definition of these sets, there exists $m \in \mathbb{N}$ such $p \in \Lambda_k(X,m)$, meaning that there exists an $m$-dominated splitting of index $k$ along the orbit of $p$. This ends the proof of Theorem~\ref{T1}.

\end{section}

\begin{section}{Proof of Theorem~\ref{T2}}

Our arguments to prove Theorem~\ref{T2} are borrowed from the ones used by Bochi, Fayad and Pujals in~\cite{BFP}. However, in the divergence-free vector fields case, we can use some $C^1$-perturbation results in order to give a more general statement than the one obtained in~\cite[Theorem 1]{BFP}. In fact, in ~\cite{BFP} they considered conservative diffeomorphisms of class $C^1$ and with derivative $\alpha$-H\"{o}lder (with $\alpha>0$) endowed with the $C^1$-topology which,  in particular, is not a complete metric space and it is not known if this subspace is $C^1$-dense in the  space of $C^1$ conservative diffeomorphisms. In our result we just need to consider $C^1$ vector fields endowed with the $C^1$-topology. 

We discuss now the three fundamental steps of the proof of Theorem~\ref{T2} and at the end of the section we complete the proof.

We recall that $X\in\mathfrak{X}_\mu^1(M)$ is $C^1$-\emph{robustly transitive} if $X^t$ has a dense orbit and any $Y\in\mathfrak{X}_\mu^1(M)$ sufficiently $C^1$-close to $X$ has also a flow with a dense orbit. It is easy to see that if $X\in \mathcal{SE}^1$, then $X$ must be $C^1$-robustly transitive. 

Let us first observe that, by ~\cite[Theorem 1.1]{BR2}, a $C^1$-stably ergodic vector field $X$ does not have singularities. As $C^s$ divergence-free vector fields are $C^1$-dense in $\mathfrak{X}_\mu^1(M)$ (Zuppa's Theorem, ~\cite{Z}), by Theorem 1.2 of ~\cite{BR2} there exists a $C^1$-dense subset of $\mathcal{SE}^1$, $\mathcal{DSE}$ whose vector fields admit a dominated splitting over $M$.

\medskip

We say that a dominated splitting $N=N^1\oplus ... \oplus N^j$ of $X$ is the \emph{finest dominated splitting} if there is no dominated splitting with more that $j$ subbundles. As is pointed out in~\cite{BFP} it is possible that the continuation of the finest dominated splitting is not the finest dominated splitting of the perturbed vector field. Hence, we say that a dominated splitting of $X \in \mathfrak{X}^{1}_{\mu}(M)$ is \emph{stably finest} if, for every $Y$ sufficiently $C^1$-close vector field in $\mathfrak{X}^{1}_{\mu}(M)$, it has a continuation which is the finest dominated splitting of $Y$. Flows with stably finest splittings are open and dense in the class of flows with dominated splitting. Given $X\in \mathcal{DSE}$ we take $X_1\in \mathfrak{X}^{1}_{\mu}(M)$ $C^1$-close to $X$ and having a stably finest dominated splitting $N=N^1(X_1)\oplus...\oplus N^k(X_1)$. We denote by 
$$\Sigma^i(X_1)=\int_M \log|P_{X_1}^{1}(p)_{|N^i(X_1)}|d\mu(p)$$ the sum of the Lyapunov exponents of the subbundle $N^i(X_1)$. In~\cite{BR} we proved the following result:

\begin{theorem}\label{removing}
Let $X_1 \in \mathfrak{X}^{1}_{\mu}(M)$ be a stably ergodic vector field having a (stably finest) dominated splitting. Then, either $\Sigma^i(X_1)\not= 0$, or else $X_1$ may be approximated, in the $C^{1}$-topology, by $X_2 \in \mathfrak{X}^{2}_{\mu}(M)$ for which $\Sigma^i(X_2) \not=0$.
\end{theorem}

We observe that the vector field $X_2$ given by the previous result is stably ergodic and has a finest dominated splitting. 

The next lemma is the final step to prove Theorem~\ref{T2}.

\begin{lemma}\label{step3}
Let $X \in \mathcal{SE}^{1}$ have a stably finest dominated splitting, $N^1(X)\oplus ...\oplus N^k(X)$. Then for all $\delta>0$, there exists $Y\in \mathfrak{X}^{1}_{\mu}(M)$ ($C^1$-close to $X$) such that, if $\lambda_j(Y)$ and $\lambda_{j+1}(Y)$ are associated to the subbundle $N^i(Y)$, then $|\lambda_j(Y)-\lambda_{j+1}(Y)|<\delta$. 
\end{lemma}
\begin{proof}
Since $\Sigma_j(\cdot)$ is an upper semicontinuous function, the set $\mathcal{C}$ of its continuity points is a residual set. By Baire theorem $\mathcal{C}$ is also dense. Let $Y \in \mathcal{SE}^{1}\cap\mathcal{C}$ be a vector field arbitrarily $C^1$-close to $X$ and let $\Lambda(j,Y)$ be a set of points such that there exists a dominated splitting $N=N^1(Y)\oplus N^2(Y)$ over the closure of $\{Y^{t}(x)\}_{t\in\mathbb{R}}$ and $\dim(N^1)=j$. Since we assume that $\lambda_j(Y)$ and $\lambda_{j+1}(Y)$ are associated to the subbundle $N^i(Y)$ and $Y$ has a finest dominated splitting, we get that $\mu(\Lambda(j,Y))=0$.

By Lemma~\ref{continuous} and using the ergodicity hypothesis, we conclude the following. 
Given any $\epsilon,\delta>0$, and $j \in \{1,\dots,d-2\}$, there exists $Z\in{\mathfrak{X}^{1}_{\mu}(M)}$, $\epsilon$-$C^1$-close to $Y$, such that 
$$ \Sigma_j(Z)<\Sigma_j(Y)- 2J_j(Y)+\delta=\Sigma_j(Y)-2(\lambda_j(Y)-\lambda_{j+1}(Y))+\delta, $$
because
\begin{eqnarray*}
J_j(Y)&=&\int_{\Gamma_j(Y,\infty)}\lambda_j(Y,p)-\lambda_{j+1}(Y,p)d\mu(p)\\
&=&\int_{M\setminus \Lambda(j,Y)}\lambda_j(Y,p)-\lambda_{j+1}(Y,p)d\mu(p) \\
&=& \lambda_j(Y)-\lambda_{j+1}(Y).
\end{eqnarray*}

Noting that $Y\in\mathcal{C}$ we can decrease $\epsilon$ if necessary and obtain that $|\Sigma_k(Z)-\Sigma_k(Y)|<\delta$. Hence, 
$$|\lambda_k(Y)-\lambda_{k+1}(Y)|<\delta,$$
and the lemma is proved.
\end{proof}

We are ready to give the proof of Theorem~\ref{T2}:\\

\emph{Openess}; Let $\mathcal{U}$ be the set of points $X\in\mathcal{SE}^1$ such that $X$ has dominated splitting $N^u\oplus N^s$ where $\dim(N^u)=j$, $\lambda_j(X)>0$ is the lowerest exponent in $N^u$ and $\lambda_{j+1}(X)<0$ is the largest  exponent in $N^s$. It is clear that any $Y\in \mathfrak{X}^{1}_{\mu}(M)$, arbitrarily close to $X$, has a dominated splitting $N^u(Y)\oplus N^s(Y)$. Moreover, since the function that gives the largest exponent $\lambda_{j+1}(\cdot)$ (in $N^s$) defined by $$Y\mapsto\inf_{n\in \mathbb{N}}\int_M \log\|P_Y^1(p)|_{N^{s}(Y,x)}\|^{1/n}$$ is upper semicontinuous, we obtain that $\lambda_{j+1}(Y)$ cannot increase ab\-rup\-tly. Therefore, if $Y$ is close enough to $X$, we have that $\lambda_{j+1}(Y)<0$. In the same way the function that gives the lowerest exponent $\lambda_{j}(\cdot)$ in $N^u$ and is defined by $$Y\mapsto\sup_{n\in \mathbb{N}}\int_M \log\mathfrak{m}(P_Y^1(p)|_{N^{u}(Y,x)})^{1/n}$$ is lower semicontinuous. Hence,  as $\lambda_{j}(Y)$ cannot decrease abruptly, we have $\lambda_{j}(Y)>0$ and we obtain that $\mathcal{U}$ is open.

\emph{Denseness}; Let $X\in\mathcal{SE}^1$. We want to prove that $X$ can be $C^1$-approximated by a vector field in $\mathcal{U}$. First we choose $X_1 \in \mathcal{DSE}$ arbitrarily close to $X$.  $X_1$ has a dominated splitting and by a small perturbation we obtain a vector field $X_2$ having a dominated splitting which is stably finest. By Theorem~\ref{removing}, $X_2$ may be approximated, in the $C^{1}$-topology, by $X_3 \in \mathfrak{X}^{2}_{\mu}(M)$ for which $\Sigma^i(X_3) \not=0$. Using Lemma~\ref{step3} we can guarantee that, for $Y$ close to $X_3$, we have all the exponents in $N^i$ very close one from the others. Since their sum is nonzero we conclude that they are all different from zero and have the same sign; hence $Y$ is a nonuniformly hyperbolic vector field and $Y \in \mathcal{U}$.

\end{section}

\begin{section}{Proof of Theorem~\ref{T3} and of Corollary~\ref{C1}}

We begin by recalling  the Pugh and Robinson $C^1$-Closing Lemma adapted to the setting of incompressible flows (see ~\cite{PR}).

The $X^t$-orbit of a recurrent point $x$ can be approximated, for a very long time $T>0$, by a periodic orbit of a $C^1$-close flow $Y$: given $r,T>0$ we can find a
$\epsilon$-$C^1$-neigh\-bor\-hood $\mathcal{U}$ of $X$ in
$\mathfrak{X}^1_\mu(M)$, a vector field $Y\in\mathcal{U}$, a periodic
orbit $p$ of $Y$ with period $\ell$ and a map
$\tau \colon [0,T]\to[0,\ell]$ close to the identity such that
 \begin{itemize}
 \item $dist\big(X^t(x),Y^{\tau(t)}(p)\big)<r$ for all
   $0\leq t\leq T$;
 \item $Y=X$ over $M\setminus\bigcup_{0\le t\le \ell}\big(
   B(p,r)\cap B(Y^t(p),r)\big)$.
 \end{itemize}
 
 Now, to prove Theorem~\ref{T3} we fix $X \in \mathcal{E}$ and observe that if the set $\mathcal{Z}$, given by Theorem~\ref{T1}, has zero Lebesgue measure then the result follows taking $N=\mathcal{D}$. So, let us assume by contradiction that the associated $\mathcal{Z}$ has positive Lebesgue measure. 
 
 By the hypothesis of Theorem~\ref{T3} there exists $\epsilon>0$ such  that any $Y \in \mathfrak{X}^1_\mu(M)$ whose $C^1$-distance of $X$ is less than $\epsilon$ has no elliptic points.
 
 By the Poincar\'e Recurrence Theorem there exists a full measure subset of $\mathcal{Z}$, $\mathcal{Z}^\prime$ consisting of recurrent points. Pick $x_0 \in \mathcal{Z}^\prime$; given $\delta>0$, as $$\lim_{t \rightarrow \pm \infty} \frac{1}{t} \log\|P_X^t(x_0)\|=0,$$ there exists $T_\delta$ such that $$\frac{1}{t} \log\|P_X^t(x_0)\|<\delta,$$ for any $|t|>T_0$. Now, if $\delta$ is small enough, applying Pugh and Robinson's $C^1$-Closing Lemma there exists $X_1 \in \mathfrak{X}^1_\mu(M)$, $\frac{\epsilon}{3}$-$C^1$-close to $X$, such that $x_0$ is a periodic orbit of $X_1$ of period $\pi_1 >T_0$, and $$\frac{1}{\pi_1} \log\|P_{X_1}^{\pi_1}(x_0)\|<2\delta.$$ By a theorem of Zuppa (\cite{Z}), and shrinking $\delta$ if necessary, there exists $X_2 \in \mathfrak{X}^4_\mu(M)$, $\frac{\epsilon}{3}$-$C^1$-close to $X_1$, such that $x_0$ is a periodic orbit of $X_2$ of period $\pi_2$ close to $\pi_1$, hence greater than $T_0$, and $$\frac{1}{\pi_2} \log\|P_{X_2}^{\pi_2}(x_0)\|<2\delta.$$ 
 
Finally, shrinking the initial $\delta$ once again if necessary, we can apply \cite[Lemma 3.2]{BR2} to obtain a vector field $X_3 \in \mathfrak{X}^1_\mu(M)$, $\frac{\epsilon}{3}$-$C^1$-close to $X_2$, hence $\epsilon$ close to $X$, such that $x_0$ is a periodic orbit of $X_3$ of period $\pi_2$, and $\frac{1}{\pi_2} \log\|P_{X_2}^{\pi_2}(x_0)\|=0$. Moreover this last perturbation can be done in such a way that $x_0$ is an elliptic point, which contradicts the hypothesis of Theorem~\ref{T3}.

\bigskip

Now, to prove Corollary~\ref{C1}, for any $k \in \mathbb{N}$ we consider the set $N_k$ consisting of orbits of $X$ that admit a $k$-dominated splitting. From Theorem~\ref{T3} it follows that $N=\cup_{k=1}^\infty {N_k}$ is a full measure set. Therefore, given any $\epsilon>0$, there is $\ell \in \mathbb{N}$ such that $\mu(\cup_{i=1}^\ell N_i) > 1-\epsilon$. Using elementary properties of dominated splittings (see for example \cite{BDV}) we conclude that there exists $\ell_1$ with $N_{\ell_1} \supset \cup_{i=1}^\ell N_i$. Therefore, as $Sing(X)=\emptyset$, it follows that the set $\overline{N_{\ell_1}}$ is compact, invariant, admits a dominated splitting and has measure greater then $1-\epsilon$, which ends the proof of the corollary.
\end{section}

\begin{section}{Main perturbation lemmas} \label{MPL}

\begin{subsection}{Local perturbations}\label{few} We start by proving a key tool that, in broad terms, assures that a time-one linear map, obtained by composing a small rotation with the linear Poincar\'e flow of a vector field, is realizable.

The next lemma, proved in ~\cite{BR} (Lemma 2.1), together with the Arbieto and Matheus Pasting Lemma (\cite{AM}) and Zuppa's Theorem (\cite{Z}), will be crucial to perform this construction.

\begin{lemma}\label{cfbt}
Given a vector field $X\in\mathfrak{X}_{\mu}^{2}(M)$ a
non-periodic point $p\in{M}$ and $t_0 \in \mathbb{R}^+$, there exists a volume-preserving $C^{2}$
diffeomorphism $\Psi$ defined in a neighborhood of the arc $\{X^t(p);\,t \in [0,t_0]\}$ such that $T=\Psi_{*}X$, where $T= \frac{\partial}{\partial x_1}$.
\end{lemma}

We fix $X\in{\mathfrak{X}^{1}_{\mu}(M)}$ and we choose a non-periodic point $p$. Let $V_p$ be a two-dimensional subspace of $N_p$; given $\xi \in \mathbb{R}$ let $\mathfrak{R}_\xi \colon V_p \rightarrow V_p$ denote the rotation of angle $\xi$. Let $V_{p}^{\perp}$ denote the orthogonal complement of $V_p$ in $N_p$ and define $R_\xi \colon N_p \rightarrow N_p$ as $R_\xi=\mathfrak{R}_\xi \oplus Id$, that is $R_\xi(v+v^{\perp})=\mathfrak{R}_{\xi}(v)+v^{\perp}$, where $v\in V_{p}$ and $v^{\perp}\in V_{p}^{\perp}$.

\begin{lemma}\label{rot1}
Given $X\in{\mathfrak{X}^{1}_{\mu}(M)}$, $\epsilon>0$ and $\kappa \in (0,1)$, there exists $\xi_{0}>0$ such that for  any $\xi\in(0,\xi_{0})$, any $p\in{X}$ (non-periodic or with period larger than one) and any two-dimensional vector space $V_p \subset N_{p}$
one has that  the time-one map $L_1=P_{X}^{1}(p)\circ R_\xi$ is an $(\epsilon,\kappa)$-realizable linear flow of length $1$ at $p$, where $R_\xi=\mathfrak{R}_\xi \oplus Id$ and  $\mathfrak{R}_\xi$ is the rotation of angle $\xi$ in $V_p$.
\end{lemma}
\begin{proof}
Fix $\epsilon$, $\kappa$, $X$, and take $\xi_0$ to be fixed latter.
Let $p \in M$ be a non-periodic point or a periodic point with period larger than one, and let $V_p$ be a two-dimensional subspace of $N_p$.

The first step consists in approximate the vector field $X$ by another one $\tilde{X} \in{\mathfrak{X}^{1}_{\mu}(M)}$ but of class $C^2$ on a small tubular neighborhood $\mathcal{V}$ of $\{X^t(p);\, t \in [\delta, 1-\delta]\}$, for some small $\delta$. Moreover, $\tilde{X}$ is rectificable in $\mathcal{V}$ and coincides with $X$ outside a small tubular neighborhood $\mathcal{W} \supset \mathcal{V}$. 

To obtain  $\tilde{X}$ we first observe that, by Zuppa's result (\cite{Z}) the set of $C^2$ divergence-free vector fields is $C^1$ dense in ${\mathfrak{X}^{1}_{\mu}(M)}$ therefore we begin by choosing a $C^2$ vector field, $X_1$, which is $\frac{\epsilon}{3}$-$C^1$-close to $X$. Now, Arbieto and Matheus Pasting Lemma (\cite{AM}) allows to get a $C^1$ vector field , $X_2$, which is $\frac{\epsilon}{3}$- $C^1$-close to $X_1$, that coincides with $X_1$ in a tubular neighborhood $\mathcal{V}$ of $\{X^t(p);\, t \in [\delta, 1-\delta]\}$ and coincides with the initial vector field $X$ outside a small tubular neighborhood $\mathcal{W}$ of $\{X^t(p);\, t \in [\frac{\delta}{2}, 1-\frac{\delta}{2}]\}$, $\mathcal{W} \supset \mathcal{V}$, such that $p$ and $X^1(p)$ do not belong to $\mathcal{W}$, for small $\delta>0$. Observe that $X$ and $X_2$ are $\frac{2\epsilon}{3}$-$C^1$-close. As $X_2$ is of class $C^2$ in $\mathcal{V}$ Lemma ~\ref{cfbt} assures that $X_2|_{\mathcal{V}}$ is $C^2$ rectificable by a change of coordinates $\Psi$; therefore we take $\tilde{X}=X_2$. 

Next step is to get $r>0$ as a function of $X, \epsilon, \kappa, \gamma$ and $p$ and then, for each $U \subset B(p,r) \subset N_p$,  perturb $X_2$ in $\mathcal{V}$ in order to obtain $K \subset U$ and $Y$ as in Definition ~\ref{rlf}. The way we get $r$ is by shrinking step by step its value along the proof.

Let $\mathcal{N}_{\tilde{X}^\delta(p)}=\Psi^{-1}(B_0)$ and $\mathcal{N}_{\tilde{X}^{t}(p)}=\tilde{X}^{t}(\mathcal{N}_{\tilde{X}^\delta(p)})$, for $t \in [-\delta, 1 -\delta]$, where $B_0$ is a ball centered at $0 \in \mathbb{R}^d$ and contained in $\{0\} \times \mathbb{R}^{d-1}$.

According to Remark~\ref{vitali} we fix any ball $U=B(p^\prime,r^\prime) \subset B(p,r) \subset \mathcal{N}_{p}$. We observe that if we prove the lemma for this particular normal section then the general case follows just by shrinking $r$.

Define $U_\delta=\{\tilde{X}^{\delta}(x); x \in U\}$. We note that, for small $r$ and $\delta$, $U_\delta$ contains a ball $\overline{B}=B(\overline{p},\overline{r})$ such that $\overline{\mu}(U_\delta \setminus \overline{B})>1-\frac{\kappa}{2}$. We let $\overline{K}=\overline{B}$ and $V_{\overline{p}}=P^{\delta}_{\tilde{X}}(p^\prime)(V_{p^\prime})$, where $V_{p^\prime}$ is the  parallel transport of $V_p$ to $p^\prime$. In Figure ~\ref{f1} we illustrate the sets we are considering.

Now Lemma 2.2 of ~\cite{BR}, which is based on Lemma ~\ref{cfbt}, guaranties the existence of a vector field $Y\in\mathfrak{X}_{\mu}^{1}(M)$
which satisfies:
\begin{enumerate}
\item [(i)] $Y$ is $\frac{\epsilon}{3}$-$C^1$-close to $\tilde{X}$;
\item [(ii)] $P_Y^{1-2\delta}(\overline{p})_{|{W_{\overline{p}}}}$ is the identity, where $W_{\overline{p}}$ is the orthogonal complement of $V_{\overline{p}}$ in $N_{\overline{p}}$;
\item [(iii)] $P_Y^{1-2\delta}(\overline{p})\cdot v=P_{\tilde{X}}^{1-2\delta}(\overline{p}) \circ R_{\xi} \cdot v$, $\forall v \in V_{\overline{p}}$, where $R_{\xi}$ is the rotation of angle $\xi$ on $V_{\overline{p}}$;
\item [(iv)] $Y=\tilde{X}$ outside the flowbox $\tilde{X}^{[\delta,1-2\delta]}(\overline{B})$.
\end{enumerate}

\begin{figure}[h]
\begin{center}
  \includegraphics[width=12cm,height=6cm]{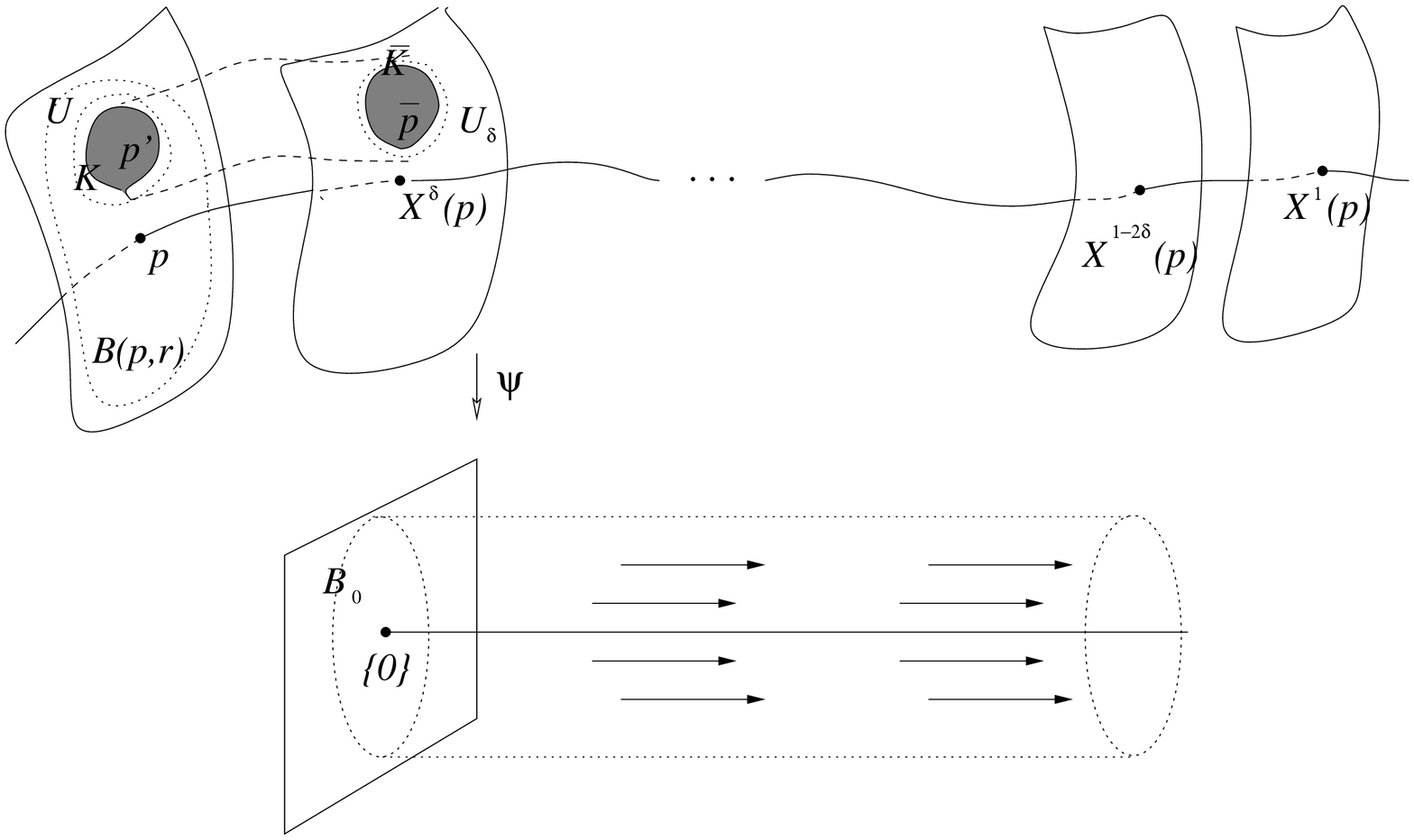}
\caption{}
\label{f1}
\end{center}
\end{figure}

Notice that condition (i) completely determines $\xi_0$.

Now we let $K=\{\tilde{X}^{-\delta}(x);\,x \in \overline{K} \}$; observe that if $r$ is small enough (and consequently $r^{\prime}$ is even smaller) then, as a consequence of Lemma ~\ref{map}, we get that $\overline{\mu}(U \setminus K)>1-\kappa$, which is condition (a) of the definition of realizable linear flow is satisfied. Also, by construction of $\tilde{X}$ and condition (i) above on $Y$, clearly imply (b) and (c) of Definition ~\ref{rlf}. 

$$P_ {Y}^1(p^\prime)=   P_ {\tilde{X}}^{\delta}(\tilde{X}^{1-\delta}(p^\prime)) \circ \left(P_{\tilde{X}}^{1-2\delta}(\tilde{X}^\delta(p^\prime)) \circ R_{\xi}\right) \circ P_ {\tilde{X}}^{\delta}(p^\prime).$$

If we take $\delta$  small enough then 
$$\|P_ {Y}^1(p^\prime)- \left(P_{\tilde{X}}^{1}(p^\prime) \circ R_{\xi}\right)\|<\frac{\gamma}{4}$$
Now, shrinking successively $r$, we get 
\begin{enumerate} 
\item $\|P_ {Y}^1(p^\prime)- \left(P_{{X}}^{1}(p^\prime) \circ R_{\xi}\right)\|<\frac{2\gamma}{4}$,
\item $\|P_ {Y}^1(p^\prime)- \left(P_{{X}}^{1}(p) \circ R_{\xi}\right)\|<\frac{3\gamma}{4}$,
\item $\|P_ {Y}^1(q)- \left(P_{{X}}^{1}(p) \circ R_{\xi}\right)\|<\gamma$, for any $q \in K$.
\end{enumerate}

\end{proof}

\begin{remark}\label{rot1b}
A completely analog proof of the previous lemma, this time considering the vector field $(-X)$ and a rotation of angle $(-\xi)$,  guaranties that the 
time-$1$ map $R_\xi\circ P_{X}^{1}(X^{-1}(p))$ is also $(\epsilon,\kappa)$-realizable linear flow of length $1$ at $X^{-1}(p)$, for small $\xi$. 
\end{remark}

If we want to prove that certain time-$2$ map is $(\epsilon, \kappa)$-realizable linear flow, which is a rotation in a fixed two-dimensional subspace $V$ of the normal space, we can try to do it by concatenating two $(\epsilon,\kappa_i)$- realizable linear flows being the first one a rotation on $V$ and the other one being an elliptical rotation on the image of $V$ by the time-one linear Poincar\'e flow. For that we need to adapt Lemma~\ref{rot1} for elliptical rotations. Note that, according to Remark~\ref{conc}, the resulting linear flow is $(\epsilon,\kappa_1+\kappa_2)$-realizable, and so the measure of the set $K$ (see Definition~\ref{rlf}) decreases. Therefore there is no hope to use directly this concatenation argument to prove that a  rotation of a given angle is $(\epsilon, \kappa)$-realizable linear flow of length $\ell$, for large $\ell$, although we need large time to get the desired rotation by a composition of rotations of small angle.
Next lemma below solves this problem.

\bigskip

Fix a regular point $p \in M$ and a two-dimensional subspace $V_p$ of $N_p$. A \emph{right cylinder} centered at $p$ is a subset of $N_p$ of the form $\mathcal{B} \oplus \mathcal{A}$, where $\mathcal{B}$ is called the basis of the cylinder and it is an ellipse contained in  $V_p$, and $\mathcal{A}$ is called the axis of the cylinder and it is a $(d-3)$-dimensional ellipsoid contained in the $V_p^\perp$, the orthogonal complement of $V_p$ in $N_p$ (see Figure~\ref{f2}).

\begin{figure}[h]
\begin{center}
  \includegraphics[width=8cm,height=4cm]{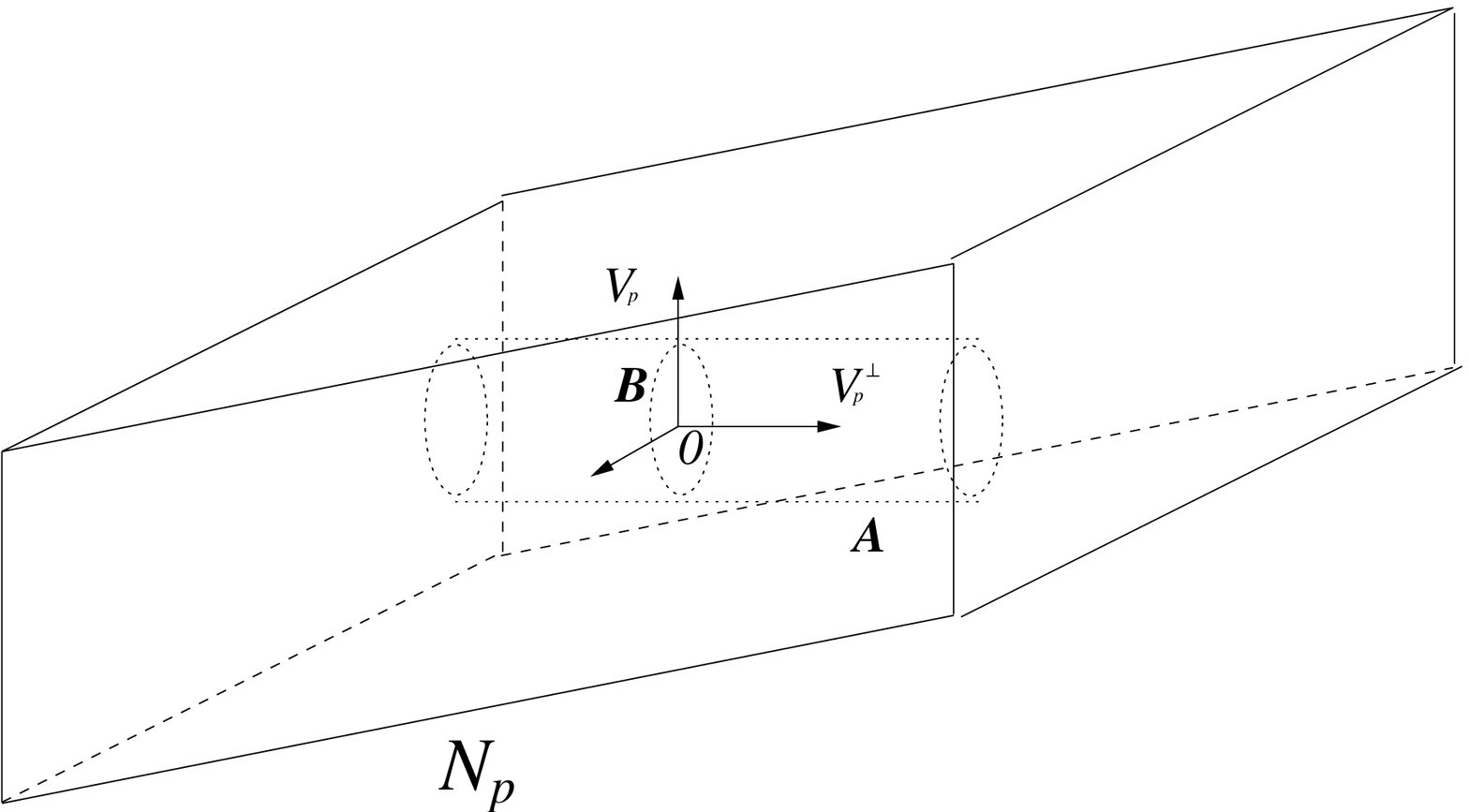}
\caption{}
\label{f2}
\end{center}
\end{figure}

Let $\mathcal{B} \oplus \mathcal{A}$ be a right cylinder centered at $p$. Let $h \colon V_p \rightarrow \mathbb{R}^2$ be an area-preserving map such that $h(\mathcal{B})$ is a disk. An elliptical rotation of angle $\theta$ of the cylinder is a map $R_\theta \colon \mathcal{B} \oplus \mathcal{A} \rightarrow \mathcal{B} \oplus \mathcal{A}$ of the form $\mathcal{E}_\theta \oplus Id$, where $h \circ \mathcal{E}_\theta \circ h^{-1}$ is a rotation of angle $\theta$.

Concerning the case of  elliptical rotations a direct adaptation of the proof of the previous lemma jointly with the strategy followed in ~\cite[Lemma 3.4]{BV} give the following result.
\begin{lemma}\label{elliptic}
Given $X \in \mathfrak{X}^1_\mu(M)$, $\epsilon>0$ and $\sigma \in ]0,1[$. There exists $\xi>0$ such that for any  $p \in M$ a non-periodic point or with period larger than one, the following holds: Let $\mathcal{B} \oplus \mathcal{A}$ be a right cylinder centered at $p$ and consider the elliptical rotation $R_\xi=\mathcal{E}_\xi \oplus Id$. For $a$, $b>0$, consider the cylinder $\mathcal{C}=b\mathcal{B} \oplus a\mathcal{A}$. There exists $\tau >1$ such that if $a \geq \tau b$ and $diam(\mathcal{C})<\epsilon$, then there exists $Y   \in \mathfrak{X}^1_\mu(M)$ satisfying
\begin{itemize}
\item  $Y=X$ outside $\mathcal{F}_X^1(p)(\hat{\mathcal{C}})$, where $\hat{\mathcal{C}}=\alpha(\mathcal{C})$ for some volume-preserving chart $\alpha$ (according to Section~\ref{lc});
\item  $\|P_Y^1(q)-P_X^1(q) \circ R_\xi\|\ll\xi$, for all $q \in \sigma \mathcal{C}$, and
\item $Y$ and $X$ are $\epsilon$-$C^1$-close.
\end{itemize}
\end{lemma}

\end{subsection}
\begin{subsection}{Large length perturbations}\label{lots}
Next lemma is the continuous-time version of ~\cite[Lemma 3.3]{BV} that allows us to realize a large concatenation of time-one perturbations under certain conditions. The proof of the lemma follows closely the arguments of Bochi-Viana's  aforementioned lemma with some extra care when dealing with the involved measure. We present the guidelines of the proof emphasizing these aspects.
\begin{lemma}\label{sequences}
Given $X \in \mathfrak{X}^1_\mu(M)$, $\epsilon>0$ and $\kappa>0$, there exists $\xi>0$ with the following properties: assume that  $p \in M$ is a non-periodic point and that for some $n \in \mathbb{N}$, and for each $j \in \{0,1,...,n-1\}$ we have 
\begin{itemize}
\item co-dimension two spaces $H_j\subset N_{X^j(p)}$ such that $H_j=P_X^j(p)(H_0)$;
\item ellipses $\mathcal{B}_j \subset (N_{X^j(p)}) /H_j$ centered at zero such that $\mathcal{B}_j=(P_X^j(p)/H_0)(\mathcal{B}_0)$ and
\item linear maps $\mathcal{E}_j \colon (N_{X^j(p)}) /H_j \rightarrow (N_{X^j(p)}) /H_j$ such that $\mathcal{E}_j(\mathcal{B}_j) \subset \mathcal{B}_j$ and $\| \mathcal{E}_j - Id \| < \xi$.
\end{itemize}
For each $j \in \{0,1,...,n-1\}$  we define $L_j=P_X^1(X^j(p)) \circ R_j$, where $R_j=\mathcal{E}_j \oplus Id$. Then $\{L_0, L_1,...,L_{n-1}\}$ is an $(\epsilon, \kappa)$-realizable linear flow of length $n$ at $p$.

\end{lemma}
\begin{proof}
Let us fix a small $\gamma>0$ according to the definition of realizable linear flow. We have to choose a sufficiently small $r>0$ such that $\mathcal{F}^n_X(p)(B(p,r))$ is a flowbox, and 
$$\|P_X^1(q)-P_X^1(X^j(p))\|.\|R_j\|<\frac{\gamma}{2},$$
for every $q\in\mathcal{P}_X^j(p)(B(p,r))$ and $j \in \{0,1,...,n-1\}$. This $r$ will be shrunk along the proof. 

We start with a ball $\mathcal{A}_0\subset H_0$ and, for $j\in\{1,...,n-1\}$, let $\mathcal{A}_j=P_{X}^{j}(p)(\mathcal{A}_0)$. By ~\cite[Lemma 3.6]{BV}, for each $j$, there exist $\hat{\tau}_j>1$ such that if $a>\hat{\tau}_j b$ then,
\begin{equation}
P_{X}^{1}(X^{j}(p))( b\mathcal{B}_j\oplus a\mathcal{A}_j)\supset  b\mathcal{B}_{j+1}\oplus \lambda a\mathcal{A}_{j+1},
\end{equation}
for some $\lambda\in (0,1)$ chosen sufficiently close to $1$ and depending on $\kappa$. This allows us to ``rightify'', at each step, the basis of the iterated cylinder and keeping almost the same $\overline{\mu}$-measure. 

We feed Lemma~\ref{elliptic} with $\epsilon$, $\sigma$ and we get $\xi_j>0$ and, for each $j$,  $\tau_j>1$ such that if $a \geq \tau_j b$ and the diameter of each cylinder $\mathcal{C}_j$ is sufficiently small, then we can realize a rotation on the basis $\mathcal{B}_j$ for $\epsilon$-close vector fields $Y_j$.

Let $\tau=\max_{j=0,\cdots,n} \{\tau_j,\hat{\tau}_j\}$ and $\xi=\min_{j=0,\cdots,n} \{\xi_j\}$. We have to consider cylinders with the axis much larger than the basis. Actually, we fix $a_0,b_0>0$ satisfying $a_0>b_0 \lambda^{-n}\tau$.

We define, for each $j$, $\mathcal{C}_j=\lambda^{j}b_0\mathcal{B}_j\oplus \lambda^{2j}a_0\mathcal{A}_j$.

By linear approximation properties (see~\cite[Lemma 3.5]{BV}), there exists $\{r_j\}_{j=0}^{n}$ such that for $\rho>0$ and $q_j\in B(X^{j}(p),r_j)\subset \mathcal{N}_{X^{j}(p)}$, if $\rho\hat{\mathcal{C}}_j+q_j\subset B(X^{j}(p),r_j)$, then 
$$X^{1}(\rho\hat{\mathcal{C}}_j+q_j)\supset \lambda P^{1}_X(X^{j}(p))(\rho\mathcal{C}_j)+X^{1}(q_j).$$

Now we decrease $r$ a little bit in order to have, for each $j$, $X^{j}(B(p,r))\subset B(X^{j}(p),r_j)$.

Let $Y$ be the divergence-free vector field and $K$ the set defined by:
\begin{itemize}
 \item $Y|_{\mathcal{F}_X^1(X^j(p))(\hat{\mathcal{C}}_j)}=Y_j$
\item $K=\mathcal{P}_Y^{-n}(Y^{n}(p))(\sigma\rho\hat{\mathcal{C}}_n+q_n)$.
\end{itemize}

Note that $$\sigma\rho\mathcal{C}_n+q_n=\sigma\rho(\lambda^{n}b\mathcal{B}_n\oplus \lambda^{2n}a\mathcal{A}_n)+q_n.$$

Let $\mu^{\ast}$ be the Lebesgue measure with a density given by the pull-back of $\overline{\mu}$ by the volume-preserving charts. Given a right cylinder $\mathcal{C}$ defined by ellipses $\mathcal{B}$ and $\mathcal{A}$, applying the Roklin Theorem (\cite{R}), locally, using the right cylinder structure we can decompose this measure as $\mu^{\ast}=\mu_{\mathcal{B}}\times \mu_{\mathcal{A}}$.

As a consequence of Lemma~\ref{map} we get $\overline{\mu}(S)\approx x(t)\overline{\mu}(P_X^t(p)(S))$ for $S\in N_p$. For a sufficiently small set $S$ the disintegration gives 

$$
\mu^\ast(S)=\int_S d\mu_{\mathcal{B}}d\mu_{\mathcal{A}}=\int_{P_X^{t}(p)(S)}\delta(\cdot) d\mu_{\mathcal{B}_t}d\mu_{\mathcal{A}_t}\approx\int_{P_X^{t}(p)(S)}\varphi(\cdot) d\mu_{\mathcal{B}_t}d\mu_{\mathcal{A}_t}\\
,$$
where $\delta(\cdot)$ is the density with respect to the $(n-1)$-volume and $\varphi$ is a density which depends only on $t$, $\mathcal{B}_t=(P_X^{t}(p)/H_0)(\mathcal{B})$ and $\mathcal{A}_t=P_X^{t}(p)(\mathcal{A})$. Note that the same holds if one  considers the restriction to $\mathcal{A}$ or  to $\mathcal{B}$. Now

\begin{eqnarray*}
\frac{\overline{\mu}(K)}{\overline{\mu}(U)}&=&\frac{\overline{\mu}(\mathcal{P}_Y^{-n}(Y^{n}(p))(\sigma\rho\hat{\mathcal{C}}_n+q_n))}{\mu^\ast(\rho b\mathcal{B}_0\oplus \rho a\mathcal{A}_0)}\\
&=&\frac{\|X(p)\|\,\overline{\mu}(\mathcal{P}_Y^{-n}(Y^{n}(p))(\sigma\rho\hat{\mathcal{C}}_n+q_n))}{\|X(p)\|\,\mu^\ast(\rho b\mathcal{B}_0\oplus \rho a\mathcal{A}_0)}\\ 
&\approx&\frac{\|X(X^{n}(p))\|\,\mu^\ast(\sigma\rho\lambda^{n}b\mathcal{B}_n\oplus\sigma\rho \lambda^{2n}a\mathcal{A}_n+q_n)}{\|X(p)\|\,\mu^\ast(\rho b\mathcal{B}_0\oplus \rho a\mathcal{A}_0)}\\
&=& x(n)\frac{\mu^\ast(\sigma\rho\lambda^{n}b(P_X^{n}(p)/H_0)(\mathcal{B}_0)\oplus \sigma\rho\lambda^{2n}aP_X^n(p)(\mathcal{A}_0)+q_n)}{\mu^\ast(\rho b\mathcal{B}_0\oplus \rho a\mathcal{A}_0)}\\
&=& x(n)\frac{\mu^\ast(\sigma\lambda^{n}(P_X^{n}(p)/H_0)(\mathcal{B}_0)\oplus\sigma \lambda^{2n}P_X^n(p)(\mathcal{A}_0)+q_n)}{\mu^\ast( \mathcal{B}_0\oplus  \mathcal{A}_0)}\\
&=& x(n)\frac{(\lambda^{n}\sigma)^{3}\mu_{\mathcal{B}_n}((P_X^{n}(p)/H_0)(\mathcal{B}_0))(\lambda^{2n}\sigma)^{d-3} \mu_{\mathcal{A}_n}(P_X^n(p)(\mathcal{A}_0))}{\mu_{\mathcal{B}}( \mathcal{B}_0)\mu_{\mathcal{A}}(\mathcal{A}_0)}\\
&\approx& x(n)\frac{(\lambda^{n}\sigma)^{3}(\lambda^{2n}\sigma)^{d-3}x(n)^{-1}\mu_{\mathcal{B}}(\mathcal{B}_0) \mu_{\mathcal{A}}(\mathcal{A}_0)}{\mu_{\mathcal{B}}( \mathcal{B}_0)\mu_{\mathcal{A}}(\mathcal{A}_0)},
\end{eqnarray*}

and we obtain $$\frac{\overline{\mu}(K)}{\overline{\mu}(U)}\approx \lambda^{2nd-3n}\sigma^d.$$
Now it is clear that $\lambda$ and $\sigma$ can be chosen such that condition (a) of Definition~\ref{rlf} is satisfied.

We are left to prove that if $q\in{K}$, then $\|P^{1}_{Y}(Y^{j}(q))-L_{j}\|<\gamma$ for
$j=0,1,...,\ell-1$. Since $L_j=P_X^1(X^j(p))\circ R_j$ we get
\begin{eqnarray*}
\|P^{1}_{Y}(Y^{j}(q))-L_{j}\|&=&\|P^{1}_{Y}(Y^{j}(q))-P_X^1(X^j(p))\circ R_j\|\\
&=&\|P^{1}_{Y}(Y^{j}(q))-P^{1}_{X}(X^{j}(q))+\\
&+& P^{1}_{X}(X^{j}(q))-P_X^1(X^j(p))\circ R_j\|\\
&\leq&\|P^{1}_{Y}(Y^{j}(q))-P^{1}_{X}(X^{j}(q))\|+\frac{\gamma}{2}<\gamma,
\end{eqnarray*}
where the last inequality is assured if we take $r$ small enough.
\end{proof}

\end{subsection}

\begin{subsection}{Proof of Proposition~\ref{exchange}}

In order to prove Proposition~\ref{exchange} we follow the strategy in~\cite[Proposition 3.1]{BV}. This proposition has an easy proof when the lack of dominated splitting comes from a small angle between the two fibers $U$ and $S$ or else comes from the fact that $S$ ``expands'' much more than $U$.

Let $X\in{\mathfrak{X}^{1}_{\mu}(M)}$, $\epsilon>0$ and $0<\kappa<1$ be given as in Proposition~\ref{exchange}. Take $\kappa'\in\left(0,\frac{1}{2}\kappa\right)$ and let $\xi_{0}=\xi_0(X,\epsilon, \kappa^\prime)$ be given by Lemma~\ref{rot1}. 
Finally, take:
\begin{equation}\label{choices}
c\geq \frac{1}{\sin^2(\xi_0)}\text{       and        } c\geq \sup_{t \in [0,2]}\left(\underset{x\in  {\mathcal{R}(X)}}{\sup}{\frac{\|P_{X}^t(x)\|}{\mathfrak{m}(P_{X}^{t}(x))}}\right)
\end{equation}

Let $\theta>0$ be such that $8\sqrt{2}c\sin\theta < \epsilon \sin^{6}(\xi_0)$. Take $m\geq 2\pi/\theta$.

Now let be given a non-periodic point $p$ and a splitting of the normal bundle at $p$, $N_{p}=U_{p}\oplus S_{p}$ such that 

$$\frac{\|P_{X}^{m}(p)|_{S_{p}}\|}{\mathfrak{m}(P_{X}^{m}(p)|_{U_{p}})}\geq \frac{1}{2},$$
and we assume that 
\begin{equation}\label{case1}
\text{there exists  } t\in[0,m]\text{ such that }\measuredangle(U_{t},S_{t})=\xi\leq\xi_{0},
\end{equation}
where we use the notation $U_{t}=P_{X}^{t}(p)(U_{p})$ and $S_{t}=P_{X}^{t}(p)(S_{p})$ for $t\in[0,m]$. 

Then we take unit vectors $s_{t}\in S_{t}$ and  $u_{t}\in U_{t}$ with $\measuredangle(s_{t},u_{t})<\xi_{0}$. If $t\in[0,m-1]$, then we use Lemma~\ref{rot1} with $V_{X^{t}(p)}=\langle s_{t},u_{t} \rangle$, where $\langle e_{1},e_{2}\rangle$ denotes the vector space spanned by $e_{1}$ and $e_{2}$, and we define the sequence:
\begin{equation}\label{prime}
L_{1}=P_{X}^{t}(p)\text{  ,   }L_{2}=P_{X}^{1}(X^{t}(p))\circ R_\xi\text{   and    }L_{3}=P_{X}^{m-t-1}(X^{t+1}(p))
\end{equation}
On the other hand, if $t\in(m-1,m]$, then we use Remark~\ref{rot1b} and we define:
\begin{equation}\label{second}
L_{1}=P_{X}^{t-1}(p)\text{  ,   }L_{2}=R_\xi\circ P_{X}^{1}(X^{t-1}(p)) \text{   and    }L_{3}=P_{X}^{m-t}(X^{t}(p))
\end{equation}
It is clear, using Remark~\ref{concat}, that the concatenation of three realizable linear flows (\ref{prime}) is an $(\epsilon,\kappa)$-realizable linear flow of length $m$ at $p$. The same works for (\ref{second}). 

In both cases we obtain vectors $\mathfrak{u}\in U_{p}\setminus \{\vec{0}\}$ and $\mathfrak{s}\in P_{X}^{m}(S_{p})\setminus \{\vec{0}\}$ such that $L_{3}\circ L_{2}\circ L_{1}(\mathfrak{u})=\mathfrak{s}$. Therefore, under the hypothesis (\ref{case1}), the proof of Proposition~\ref{exchange} is completed.

\bigskip

Now we assume that there exist $r,t\in\mathbb{R}$ with $0\leq r+t\leq m$ such that:
\begin{equation}\label{case2}
\frac{\|P_{X}^{r}(X^{t}(p))|_{S_{t}}\|}{\mathfrak{m}(P_{X}^{r}(X^{t}(p))|_{U_{t}})}\geq c. 
\end{equation}
Observe that, by the choice of $c$, we have that $r\geq 2$.
 
We choose unit vectors:
\begin{itemize}
\item $s_{t}\in S_{t}$ such that $\|P_{X}^{r}(X^{t}(p))\cdot s_{t}\|=\|P_{X}^{r}(X^{t}(p))|_{S_{t}}\|$;  \item $u_{t}\in U_{t}$ such that $\|P_{X}^{r}(X^{t}(p))\cdot u_{t}\|=\mathfrak{m}(P_{X}^{r}(X^{t}(p))|_{U_{t}})$;
\item $u_{t+r}=\frac{P_{X}^{r}(X^{t}(p))\cdot u_{t}}{\|P_{X}^{r}(X^{t}(p))\cdot u_{t}\|}\in U_{t+r}$ and
\item $s_{t+r}=\frac{P_{X}^{r}(X^{t}(p))\cdot s_{t}}{\|P_{X}^{r}(X^{t}(p))\cdot s_{t}\|}\in S_{t+r}$.
\end{itemize}
The vector $\hat{u}_{t}= u_{t}+\sin(\xi_{0})s_{t}$ is such that $\measuredangle(\hat{u}_{t},u_{t})<\xi_{0}$ so we consider $L_{2}=P_{X}^{1}(X^{t}(p))\circ R_{2}$, where $R_{2}$ is the rotation on $V_{X^{t}(p)}=\langle s_{t},u_{t} \rangle$, which sends $u_{t}$ into $\frac{\hat{u}_{t}}{\|\hat{u}_{t} \|}$. 

Let $$\varrho=\frac{\|P_{X}^{r}(X^{t}(p))\cdot u_{t}\|}{\sin\xi_{0} \|P_{X}^{r}(X^{t}(p))\cdot s_{t}\|}.$$ Let us define a vector in $N_{X^{t+r}(p)}$ by $\hat{s}_{t+r}=\varrho u_{t+r}+s_{t+r}$. We have that,
\begin{eqnarray*}
P_{X}^{r}(X^{t}(p))\cdot\hat{u}_{t}&=& P_{X}^{r}(X^{t}(p))\cdot  u_{t}+\sin(\xi_{0})P_{X}^{r}(X^{t}(p))\cdot s_{t} \\
&=&P_{X}^{r}(X^{t}(p))\cdot  u_{t}+\frac{\|P_{X}^{r}(X^{t}(p))\cdot u_{t}\|}{\varrho\|P_{X}^{r}(X^{t}(p))\cdot s_{t}\|}P_{X}^{r}(X^{t}(p))\cdot s_{t}\\
&=& \frac{1}{\varrho}\|P_{X}^{r}(X^{t}(p))\cdot u_{t}\|.(\varrho u_{t+r}+s_{t+r})\\
&=& \frac{1}{\varrho}\|P_{X}^{r}(X^{t}(p))\cdot u_{t}\|.\hat{s}_{t+r}.
\end{eqnarray*}
It follows from ~(\ref{choices}),~(\ref{case2}) and definition of $\varrho$, $u_{t}$ and $s_{t}$, that 
$$\varrho=\frac{\mathfrak{m}(P_{X}^{r}(X^{t}(p))|_{U_{t}})}{\|P_{X}^{r}(X^{t}(p))|_{S_{t}}\|\sin\xi_{0}}\leq \frac{1}{c\sin\xi_{0}}<\sin\xi_{0}.$$ 
So, $\measuredangle(s_{t+r},\hat{s}_{t+r})< \xi_{0}$. Let $L_{4}=R_{4}\circ P_{X}^{1}(X^{t+r-1}(p))$, where $R_{4}$ acts in $V_{X^{t+r}(p)}=\langle s_{t+r},\hat{s}_{t+r} \rangle$ and sends $\frac{\hat{s}_{t+r}}{\|\hat{s}_{t+r}\|}$ into $s_{t+r}$. By Remark~\ref{rot1b} we obtain that $L_{4}$ is a realizable linear flow of length $1$ at $X^{t+r-1}(p)$.  Now we concatenate as follows:
$$N_{p}\overset{L_{1}}{\longrightarrow}N_{X^{t}(p)}\overset{L_{2}}{\longrightarrow}N_{X^{t+1}(p)}\overset{L_{3}}{\longrightarrow}N_{X^{t+r-1}(p)}\overset{L_{4}}{\longrightarrow}N_{X^{t+r}(p)}\overset{L_{5}}{\longrightarrow}N_{X^{m}(p)},$$
where $L_{1}=P_{X}^{t}(p)$, $L_{3}=P_{X}^{r-2}(X^{t+1}(p))$ and $L_{5}=P_{X}^{m-t-r}(X^{t+r}(p))$, and $L_3=Id$ if $r=2$. In this way, applying Lemma ~\ref{rot1} twice and recalling that $L_2$ and $L_4$ are $(\epsilon,\kappa^\prime)$-realizable linear flows of length $1$, we obtain an $(\epsilon,\kappa)$-realizable linear flow of length $m$ at $p$ such that $L_{5}\circ L_{4}\circ L_{3}\circ L_{2}\circ L_{1}(\mathfrak{u})=\mathfrak{s}$, where $\mathfrak{u}=P_X^{-t}(X^t(p))\cdot u_t$ and $\mathfrak{s}$ is a vector co-linear with $P_X^{m-t}(X^t(p)) \cdot s_t$.

So, assuming (\ref{case2}), the proof of Proposition~\ref{exchange} is done.

\bigskip

Now we shall finish the proof of Proposition~\ref{exchange} by considering the last case, that is, when we have  that conditions (\ref{case1}) and (\ref{case2}) are not simultaneously satisfied, that is:
\begin{equation}\label{case3a}
\text{for all  } t\in[0,m]\text{ we have that }\measuredangle(U_{t},S_{t})>\xi_{0},
\end{equation}
and for all $r,t\in\mathbb{R}$ with $0\leq r+t\leq m$ we have:
\begin{equation}\label{case3b}
\frac{\|P_{X}^{r}(X^{t}(p))|_{S_{t}}\|}{\mathfrak{m}(P_{X}^{r}(X^{t}(p))|_{U_{t}})}< c. 
\end{equation}
The conditions (\ref{case3a}) and (\ref{case3b}) together with~\cite[Lemma 3.8]{BV} allows us to conclude that for every $t\in [0,m]$ we have

\begin{equation}\label{lemmaBV}
\frac{\|P_X^{t}(q)/H_0\|}{\mathfrak{m}(P_X^{t}(q)/H_0)} \leq \frac{8c}{\sin^{6}(\xi_0)}.
\end{equation}

We define unit vectors $u\in U_0$ and $s\in S_0$ such that 
\begin{itemize}
 \item $\|P_X^m(p)\cdot u\|=\mathfrak{m}(P_X^m(p)|_{U_0})$ and
 \item $\|P_X^m(p)\cdot s\|=\|P_X^m(p)|_{S_0})\|$. 
\end{itemize}
Now define $s^{\prime}=\frac{P_X^m(p)\cdot s}{\|P_X^m(p)\cdot s\|}\in S_m$. Like in ~\cite[Lemma 3.8]{BV} we consider $G_0=U_0\cap u^{\perp}$, $G_t=P_X^t(p)(G_0)\subseteq U_t$ for $t\in[0,m]$, $F_m=S_m\cap (s^{\prime})^{\perp}$ and $F_t=P_X^{t-m}(p)(F_m)\subseteq S_t$ for $t\in[0,m]$. We consider unit vectors $v_t\in U_t\cap G_t^{\perp}$ and $w_t\in S_t\cap F_t^{\perp}$ for $t\in[0,m]$.

We continue defining some useful objects; let, for $t\in[0,m]$, $H_t=G_t\oplus F_t$ and $I_t=v_t\cdot\mathbb{R}\oplus w_t\cdot\mathbb{R}$. Let $\mathcal{B}_0\subset N_p/H_0$ be a ball and, for $t\in[0,m]$, $\mathcal{B}_t=(P_X^t(p)/H_0)(\mathcal{B}_0)$.

Since $m\geq 2\pi/\theta$ we take $\{\theta\}_{j=0}^{m-1}$ such that $|\theta_j|\leq \theta$ for all $j$ and $\sum_{j=0}^{m-1}\theta_j=\measuredangle(v_0+H_0,w_0+H_0)$.

Let $\mathcal{E}_{\theta_j}$ be the rotation of angle $\theta_j$ in $N_p/H_0$ and define:

$$R_{\theta_j}=(P_X^j(p)/H_0)\circ \mathcal{E}_{\theta_j}\circ (P_X^j(p)/H_0)^{-1}.$$

Clearly we have $R_{\theta_j}(\mathcal{B}_j)=\mathcal{B}_j$ for all $j$. Moreover, by (\ref{lemmaBV}) we obtain,

$$\|R_{\theta_j}-Id\|\leq \frac{\|P_X^j(p)/H_0\|}{\mathfrak{m}(P_X^j(p)/H_0)}\|\mathcal{E}_{\theta_j}-Id\|\leq \frac{8c}{\sin^6(\xi_0)}\sqrt{2}\sin\theta.$$

For each $j \in \{0,1,...,n-1\}$  we define $L_j=P_X^1(X^j(p)) \circ R_j$. Then, by Lemma~\ref{sequences}, $\{L_0, L_1,...,L_{n-1}\}$ is an $(\epsilon, \kappa)$-realizable linear flow of length $n$ at $p$. 

Notice that $L_j|_{H_j}=P_X^1(X^j(p))|_{H_j}$ and 
$$L_j/H_j=(P_X^1(X^j(p))/H_j)\circ R_{\theta_j}=(P_X^{j+1}(p)/H_0)\circ\mathcal{E}_{\theta_j}\circ (P_X^j(p)/H_0)^{-1}.$$ 

By composing the sequences we obtain,

$$(L_{m-1}\circ ...\circ L_0)/H_0=(P_X^m(p)/H_0)\circ \mathcal{E}_{\theta_j}\circ ...\circ \mathcal{E}_{\theta_1}.$$

Applying to $v_0+H_0$ we get 
$$(L_{m-1}\circ ...\circ L_0)(v_0+H_0)=(P_X^m(p)/H_0)(w_0+H_0)=P_X^m(p)(w_0)+H_m.$$

Note that $H_m=G_m\oplus I_m$. Then, $$(L_{m-1}\circ ...\circ L_0)(v_0)=P_X^m(p)(w_0)+g_m+i_m,$$ where $g_m\in G_m$ and $i_m\in I_m$.

Let $g_0=(P_X^m(p))^{-1}(g_m)\in G_0\subset H_0\cap U_0$. 

Clearly, $(L_{m-1}\circ ...\circ L_0)(g_0)=g_m$, therefore the vector $v_0-g_0\in U_0$ is such that

$$(L_{m-1}\circ ...\circ L_0)(v_0-g_0)=P_X^m(p)(w_0)+i_m\in S_m,$$
and Proposition~\ref{exchange} is proved.

\end{subsection}

\end{section}

\section*{Acknowledgements}

MB was partially supported by Funda\c c\~ao para a Ci\^encia e a Tecnologia, SFRH/BPD/
20890/2004. 
JR was partially supported by Funda\c c\~ao para a Ci\^encia e a Tecnologia through the Program FCT-POCTI/MAT/61237/2004. 


\end{document}